\documentclass{article}
\usepackage[a4paper, margin=3cm, bottom=4cm]{geometry}
\usepackage[utf8]{inputenc}
\usepackage[T1]{fontenc}
\usepackage{lmodern}

\usepackage[
    backend=biber,
    style=alphabetic, maxalphanames=5, 
    maxbibnames=100, maxsortnames=100, 
    sorting=nyt,
    natbib=true,
    url=false,doi=false,isbn=false,eprint=false
]{biblatex}
\usepackage[pdftex]{graphicx} 
\usepackage{subcaption}
\graphicspath{ {./images/} }
\usepackage{enumitem}
\usepackage[normalem]{ulem}
\usepackage[toc,page]{appendix}

\usepackage[bookmarks, bookmarksnumbered,
			colorlinks=true,
            linkcolor = blue,
            urlcolor  = blue,
            citecolor = teal,
            hypertexnames=false  
]{hyperref}

\newcommand\fnsurl[1]{{\footnotesize\url{#1}}}

\usepackage{myquicksetup}
\numberwithin{definition}{section}
\numberwithin{theorem}{section}
\numberwithin{corollary}{section}
\numberwithin{proposition}{section}
\numberwithin{lemma}{section}
\numberwithin{claim}{section}
\numberwithin{fact}{section}
\numberwithin{remark}{section}
\numberwithin{example}{section}
\numberwithin{equation}{section}
\mathtoolsset{showonlyrefs} 

\usepackage{mylivemacros}

\usepackage{sinkhorn_small} 

\usepackage{subfiles} 

\newif\ifextended  
\extendedtrue

\title{Almost-sharp $O(k^{-1} \log k)$ convergence rate for the Sinkhorn algorithm in the asymptotically scalable case}
\date{\today}

\author{Guillaume Wang\footnote{Courant Institute School, New York
University~ \texttt{guillaume.wang@nyu.edu}}}

\hypersetup{
    pdftitle={Almost-sharp O(k-1 log k) convergence rate for the Sinkhorn algorithm in the asymptotically scalable case},
    pdfauthor={Guillaume Wang},
    pdfkeywords={Sinkhorn algorithm, matrix scaling, asymptotically scalable}
}

\begin{document}
\maketitle


\begin{abstract}
    We prove that the Sinkhorn algorithm converges at a rate of $O(k^{-1} \log k)$ in $\ell_1$-norm marginal error, in the asymptotically scalable case. This almost closes the gap between the lower bound $\Omega(k^{-1})$ \cite{qu2025sinkhorn} and the previously best known upper bound $O(k^{-1/2})$ \cite{leger2021gradient},
    and generalizes the analysis for the positive case by
    \cite{dvurechensky2018computational}.
\end{abstract}

\section{Introduction} \label{sec:intro}

The purpose of this paper is to prove that the Sinkhorn algorithm converges at a rate of $O(k^{-1} \log k)$ in $\ell_1$-norm marginal error, in the asymptotically scalable case.

We start by introducing the objects studied throughout this paper, in the next section.
The reader familiar with the Sinkhorn algorithm may first skip ahead to \autoref{subsec:intro:sota_contrib}.

\subsection{Problem setup} \label{subsec:intro:setup}

Let $\mu \in \Delta_m$ and $\nu \in \Delta_n$, where $\Delta_m$ denotes the probability simplex in dimension $m$, such that $\mu_{\min} = \min_i \mu_i, \nu_{\min} = \min_j \nu_j > 0$.
Let $C \in (\RR \cup \{\infty\})^{m \times n}$ and $\EEE = \left\{ (i,j);~ C_{ij} < \infty \right\}$, and suppose that the bipartite graph $(\{1 \dots m \} \sqcup \{1 \dots n\}, \EEE)$ 
is connected.
Let $\tau>0$ and let the function $\Psi: \RR^m \times \RR^n \to \RR$ be defined by
\begin{equation}
    \Psi(f, g) =
    \tau \log \sum_{ij} e^{\left[ -C_{ij} + f_i + g_j \right]/\tau} \mu_i \nu_j
    - \mu^\top f - \nu^\top g
\end{equation}
with the convention that $\exp(-\infty) = 0$, i.e., the sum can equivalently be taken over the $(i,j) \in \EEE$.

For a given initial pair $(f^0, g^0)$, typically $(0,0)$, we call iterates of the Sinkhorn algorithm the sequence $(f^k, g^k)_{k \geq 0}$ defined by the update rule
\begin{align}
    &\text{for $k$ even,}~~~
    f^{k+1} = f[g^k]
    \quad \text{and} \quad
    g^{k+1} = g^k \\
    &\text{for $k$ odd,}~~~~
    f^{k+1} = f^k
    \quad\quad \text{and} \quad
    g^{k+1} = g[f^k]
\end{align}
where
\begin{equation}
    \forall g \in \RR^n,~
    f[g]_i = -\tau \log \sum_j\, e^{\left[ -C_{ij} + g_j \right]/\tau} \nu_j
    \qquad \text{and} \qquad
    \forall f \in \RR^m,~
    g[f]_j = -\tau \log \sum_i\, e^{\left[ -C_{ij} + f_i \right]/\tau} \mu_i,
\end{equation}
still with the convention that $\exp(-\infty) = 0$.
Note that 
$\forall g, f[g] \in \argmin \Psi(\cdot, g)$
and
$\forall f, g[f] \in \argmin \Psi(f, \cdot)$.
Also define
\begin{equation}
    \forall f, g,~~
    \pi[f,g]_{ij}
    = \frac{1}{Z(f,g)} e^{\left[ -C_{ij} + f_i + g_j \right]/\tau} \mu_i \nu_j
    \quad \text{where} \quad
    Z(f,g) = \sum_{i'j'} e^{\left[ -C_{i'j'} + f_{i'} + g_{j'} \right]/\tau} \mu_{i'} \nu_{j'}
\end{equation}
and set $\pi^k = \pi[f^k, g^k] \in \Delta_{m \times n}$ for all $k \geq 0$.
Note that by explicit computations,
\begin{equation}
    \forall g,~ Z(f[g], g) = 1
    \qquad \text{and} \qquad
    \forall f,~ Z(f, g[f]) = 1,
\end{equation}
so we have $Z(f^k, g^k) = 1$
for all $k \geq 1$ (but not for $k=0$ in general).

Denote by $\Hdiv{\pi}{\pi'} = \sum_I \pi_I \log \frac{\pi_I}{\pi'_I}$ the relative entropy between any discrete probability distributions.
For any $\pi \in \Delta_{m \times n}$, let
\begin{equation}
    (X_\sharp \pi)_i = \sum_j \pi_{ij}
    \qquad\text{and}\qquad
    (Y_\sharp \pi)_j = \sum_i \pi_{ij}.
\end{equation}
We identify probability distributions on $\EEE$ to probability distributions on $\{1 \dots m\} \times \{1 \dots n\}$ by setting
$\forall \pi \in \Delta_\EEE, \forall (i,j) \not\in \EEE, \pi_{ij} = 0$.
The optimality metric we use to measure the eventual convergence of the Sinkhorn algorithm is the $\ell_1$-norm error of the marginals,
as is standard in the literature; 
that is,
\begin{equation}
    E_k = \norm{X_\sharp \pi^k - \mu}_1 + \norm{Y_\sharp \pi^k - \nu}_1.
\end{equation}

\begin{remark}
    Since $f^1 = f[g^0]$ and $g^1=g^0$, the Sinkhorn iterates $(f^k, g^k)$ do not depend on $f^0$, except at the iteration $k=0$ of course. 
    The quantities $f^0, \pi^0, Z^0$ were introduced only for ease of presentation and do not play any role.
\end{remark}

\begin{remark}
    Almost all of the statements and derivations in this paper would hold without change, and all could be adapted easily, if $\Psi$ is replaced by $\tPsi(f, g) = \tau \sum_{ij} e^{[-C_{ij}+f_i+g_j]/\tau} \mu_i \nu_j - \tau - \mu^\top f - \nu^\top g$ everywhere.
\end{remark}

\begin{remark}
    The optimization problem $\min_{f,g} \Psi(f,g)$ is the convex dual of the entropy-regularized optimal transport (EOT) problem
    \begin{equation} 
        \min_{\pi \in \Delta_{\EEE}} \sum_{ij} C_{ij} \pi_{ij} + \tau \Hdiv{\pi}{\mu \otimes \nu}
        ~~~~\text{subject to}~~~~
        \begin{cases}
            X_\sharp \pi = \mu \\
            Y_\sharp \pi = \nu.
        \end{cases}
    \end{equation}
\end{remark}

\subsection{State of the art and contributions} \label{subsec:intro:sota_contrib}

Let $A \in \RR_+^{m \times n}$ be the matrix defined by $A_{ij} = e^{-C_{ij}/\tau} \mu_i \nu_j$.
We say that
\begin{itemize}
    \item $A$ is asymptotically $(\mu, \nu)$-scalable if $\inf_{f, g} \Psi(f,g) > -\infty$, or equivalently, if the EOT problem admits a feasible solution.
    \item $A$ is exactly $(\mu, \nu)$-scalable if $\Psi$ attains its infimum at some $(f^*, g^*)$ (which may not be unique), or equivalently, if the EOT problem admits a feasible
    solution that assigns a positive mass to all $(i,j) \in \EEE$.
    \item $A$ is positive if $\forall i,j, A_{ij}>0$, or equivalently, if $\forall i,j, C_{ij}<\infty$.
\end{itemize}
These definitions are consistent with the usual terminology in matrix scaling, see 
\cite[Theorems~4.1, 4.2]{idel2016review}.
Throughout this paper, we use ``scalable'' as an abbreviation for ``$(\mu, \nu)$-scalable''.
It is known that
\begin{equation}
    \text{$A$ is positive}
    ~~\implies~~
    \text{$A$ is exactly scalable}
    ~~\implies~~
    \text{$A$ is asymptotically scalable}.
\end{equation}

Several convergence analyses of the Sinkhorn algorithm are available for the case where $A$ is positive, in the literature on computational optimal transport.
The same is true for the case where $A$ is only assumed exactly scalable, in the literature on matrix scaling.
For both of these cases, four types of convergence bounds exist:
\begin{itemize}
    \item Global exponential convergence.
    See \cite{franklin1989scaling} for the case where $A$ is positive (the contraction constant therein is sharp) and \cite{qu2025sinkhorn} for the case where $A$ is exactly scalable.
    \item Local exponential convergence.
    See \cite{soules1991rate,knight2008sinkhorn,qu2025sinkhorn} for the case where $A$ is exactly scalable (the exponential rate in the last reference is sharp).
    \item Slow polynomial bound in $E_k \leq O(1/\sqrt{k})$.
    See \cite{altschuler2017near} for the case where $A$ is positive and \cite{leger2021gradient} for the case where $A$ is exactly scalable. See also \cite{kalantari2008complexity} for an earlier result with looser constants.
    \item Fast polynomial bound in $E_k \leq O(1/k)$. See \cite{dvurechensky2018computational} for the case where $A$ is positive
    and \cite{ghosal2025convergence} for the case where $A$ is exactly scalable.
    The latter case can also be treated by combining \cite{dvurechensky2018computational} and \cite{kalantari2008complexity}, and this yields a bound with more explicit constants; see \autoref{sec:warmup} for details.
\end{itemize}
We note that while the exponential convergence bounds may seem strictly stronger than the polynomial bounds,
the former also scale exponentially in $1/\tau$ for small $\tau$, instead of only polynomially for the latter.

In this paper, we focus on the case where $A$ is only asymptotically scalable (and not exactly scalable).
Let us summarize the known results for this case.
\begin{itemize}
    \item It is known since \cite{sinkhorn1967concerning} that we have the qualitative convergence $E_k = o(1)$.
    \item \cite[Proposition~3]{qu2025sinkhorn} shows that necessarily $E_k \geq \Omega(1/k)$. \cite[page~19]{soules1991rate} and \cite{achilles1993implications} give an explicit example where $E_k = \Theta(1/k)$.
    \item The only quantitative convergence upper bound currently available is $E_k \leq O(1/\sqrt{k})$, shown in \cite{leger2021gradient}.
    The same rate was attained independently in \cite{vanapeldoorn2020quantum}, building upon \cite{chakrabarty2021better}.
\end{itemize}
In this paper, we significantly reduce the gap between upper and lower bounds, by proving the almost-sharp convergence upper bound
$E_k \leq O(\log k/k)$.
Formally, we show the following.

\begin{theorem}
    Suppose that $A$ is asymptotically scalable.
    There exist constants $B_1, B_2>0$ dependent only on $\mu, \nu$, and $\EEE$ such that,
    if the Sinkhorn algorithm is initialized with $g^0 = 0$,
    then
    $E_k = \norm{X_\sharp \pi^k - \mu}_1 + \norm{Y_\sharp \pi^k - \nu}_1$
    is bounded as 
    \begin{equation}
        \forall k \geq B_1,~~
        E_k 
        \leq \frac{B_2}{k} \left( 1 + \mathrm{osc}(C)/\tau + \log k \right)
    \end{equation}
    where $\mathrm{osc}(C) = \max_{(i,j) \in \EEE} C_{ij} - \min_{ij} C_{ij}$.
\end{theorem}

See \autoref{thm:fastasymp:fastasymp} for the full statement of our result, including explicit constants.

\begin{remark}
    Whether or not $A$ is asymptotically $(\mu, \nu)$-scalable, or exactly $(\mu, \nu)$-scalable, or positive, only depends on $\mu, \nu$, and $\EEE$ rather than on the specific coefficients of $A$ (or of $C$).
\end{remark}

\begin{remark}
    In the case where $A$ is not even asymptotically scalable, the Sinkhorn algorithm is still well-defined, and its convergence behavior is detailed in \cite{baradat2024convergence}.
    In particular, for the purpose of quantifying the convergence rate, the proof of Proposition~5.3 therein allows to reduce the study to the asymptotically scalable case.
\end{remark}

\begin{remark}
    The Sinkhorn algorithm is well-defined as soon as the bipartite graph $(\{1 \dots m \} \sqcup \{1 \dots n\}, \EEE)$ has no isolated vertex.
    Throughout this paper we make the additional assumption that this graph is connected, but this is without loss of generality.
    Indeed, if it is not connected, let $I_1 \sqcup J_1, ..., I_P \sqcup J_P$ denote its connected components
    and let
    $\mu^{(p)} = (\mu_i)_{i \in I_p}$,
    $\nu^{(p)} = (\nu_j)_{j \in J_p}$, and
    $C^{(p)} = (C_{ij})_{i \in I_p, j \in J_p}$.
    Then running the Sinkhorn algorithm on $(\mu, \nu, C)$ is equivalent to running it on each of the $(\mu^{(p)}, \nu^{(p)}, C^{(p)})$ independently, in the sense that (the relevant components of) the iterates coincide.
    Moreover, requiring $A$ to be asymptotically scalable implies that 
    for each $p$, $\sum_{I_p} \mu_i = \sum_{J_p} \nu_j$ and
    $A^{(p)} = (A_{ij})_{i \in I_p, j \in J_p}$ is also $(\restr{\mu}{I_p}, \restr{\nu}{J_p})$-asymptotically scalable.
\end{remark}

\subsection{Preliminaries} \label{subsec:intro:prelim}

We start by recalling some known facts about the Sinkhorn algorithm.

\begin{lemma} \label{lm:intro:marginals_match}
    For any $k \geq 0$, if $k$ is even then
    $X_\sharp \pi^{k+1} = \mu$,
    and if $k$ is odd then
    $Y_\sharp \pi^{k+1} = \nu$.
\end{lemma}

\begin{proof}
    Let $k \geq 0$ even.
    We have by definition
    \begin{equation}
        (X_\sharp \pi^{k+1})_i
        = \sum_j e^{[-C_{ij} + f^{k+1}_i + g^{k+1}_j]/\tau} \mu_i \nu_j
        = \sum_j e^{[-C_{ij} + g^k_j]/\tau} \frac{1}{\sum_{j'} e^{[-C_{ij'} + g^k_{j'}]/\tau} \nu_{j'}} \mu_i \nu_j
        = \mu_i.
    \end{equation}
    The statement for $k$ odd follow similarly.
\end{proof}

\begin{lemma} \label{lm:intro:fk+1-fk}
    For any $k \geq 1$, if $k$ is even then
    $\forall i, f^{k+1}_i 
    = f^k_i - \tau \log \frac{(X_\sharp \pi^k)_i}{\mu_i}$,
    and if $k$ is odd then
    $\forall j, g^{k+1}_j 
    = g^k_j - \tau \log \frac{(Y_\sharp \pi^k)_j}{\mu_j}$.
\end{lemma}

\begin{proof}
    Let $k \geq 1$ even. We have
    \begin{equation}
        f^{k+1}_i - f^k_i
        = -\tau \log \sum_j e^{[-C_{ij} + f^k_i + g^k_j]/\tau} \nu_j
        = -\tau \log \frac{(X_\sharp \pi^k)_i}{\mu_i}
    \end{equation}
    since $\pi^k_{ij} = e^{[-C_{ij} + f^k_i + g^k_j]/\tau} \mu_i \nu_j$.
    The statement for $k$ odd follows similarly.
\end{proof}

\begin{lemma}[{\cite[Lemma~2]{altschuler2017near}}] \label{lm:intro:Psik-Psik+1}
    For any $k \geq 1$,
    \begin{equation}
        \Psi(f^k, g^k) - \Psi(f^{k+1}, g^{k+1})
        = \tau \Hdiv{\mu}{X_\sharp \pi^k} + \tau \Hdiv{\nu}{Y_\sharp \pi^k}.
    \end{equation}
\end{lemma}

\begin{proof}
    Since $Z(f^k, g^k) = 1$ then
    $\Psi(f^k, g^k) = -\mu^\top f^k - \nu^\top g^k$,
    and so
    $\Psi(f^k, g^k) - \Psi(f^{k+1}, g^{k+1})
    = -\mu^\top (f^k-f^{k+1}) - \nu^\top (g^k-g^{k+1})$.
    Suppose $k$ is even, then
    \begin{equation}
        \Psi(f^k, g^k) - \Psi(f^{k+1}, g^{k+1})
        = \mu^\top (f^{k+1}-f^k) + 0
        = \tau \sum_i \mu_i \log \frac{(X_\sharp \pi^k)_i}{\mu_i}
        = \tau \Hdiv{\mu}{X_\sharp \pi^k}
    \end{equation}
    by \autoref{lm:intro:fk+1-fk},
    and $\Hdiv{\nu}{Y_\sharp \pi^k} = 0$ by \autoref{lm:intro:marginals_match}.
    Hence the equality for $k$ even, and the case of $k$ odd follows similarly.
\end{proof}

\begin{lemma} \label{lm:intro:nablaPsi}
    The function $\Psi$ is (jointly) convex.
    Moreover for any $f \in \RR^m, g \in \RR^n$,
    \begin{equation}
        \forall i,~
        \nabla_{f_i} \Psi(f, g)
        = (X_\sharp \pi[f,g])_i - \mu_i
        \qquad \text{and} \qquad
        \forall j,~
        \nabla_{g_j} \Psi(f,g)
        = (Y_\sharp \pi[f,g])_j - \nu_j.
    \end{equation}
\end{lemma}

\begin{proof}
    By definition of $\Psi(f, g) = \tau \log \sum_{ij} e^{[-C_{ij}+f_i+g_j]/\tau} \mu_i \nu_j - \mu^\top f - \nu^\top g$ and convexity of the log-sum-exp function,
    $\Psi$ is convex.
    Moreover,
    \begin{equation}
        \forall i,~
        \nabla_{f_i} \Psi(f,g)
        = \frac{\sum_j e^{[-C_{ij}+f_i+g_j]/\tau} \mu_i \nu_j}{\sum_{i'j'} e^{[-C_{i'j'}+f_{i'}+g_{j'}]/\tau} \mu_{i'} \nu_{j'}}
        - \mu_i
        = (X_\sharp \pi[f,g])_i - \mu_i.
    \end{equation}
    The statement for $\nabla_{g_j} \Psi$ follows similarly.
\end{proof}

The following remarkable monotonicity property first appeared, to our knowledge, in \cite{nutz2021introduction}.
This reference attributes the result to \cite{leger2021gradient}, which however only explicitly shows the monotonicity of $\Hdiv{X_\sharp \pi^k}{\mu}$ and $\Hdiv{Y_\sharp \pi^k}{\nu}$.

\begin{lemma} \label{lm:intro:sqgrad_decr}
    The sequences
    $\left( \Hdiv{X_\sharp \pi^k}{\mu} \right)_{k \in 2\NN}$,
    $\left( \Hdiv{\mu}{X_\sharp \pi^k} \right)_{k \in 2\NN}$,
    $\left( \Hdiv{Y_\sharp \pi^k}{\nu} \right)_{k \in 2\NN+1}$,
    and 
    $\left( \Hdiv{\nu}{Y_\sharp \pi^k} \right)_{k \in 2\NN+1}$
    are non-increasing.
\end{lemma}

\begin{proof}
    See \cite[Proposition~6.10]{nutz2021introduction}.
\end{proof}

The following non-expansiveness result is classical
in computational optimal transport
and in the context of the Hilbert projective metric.

\begin{lemma} \label{lm:intro:varseminorm_nonexpansive}
    We have
    \begin{equation}
        \forall g,g' \in \RR^n,~
        \norm{f[g] - f[g']}_{\var}
        \leq \norm{g-g'}_{\var}
        ~\quad \text{and} \quad~
        \forall f,f' \in \RR^m,~
        \norm{g[f] - g[f']}_{\var}
        \leq \norm{f-f'}_{\var}
    \end{equation}
    where $\norm{\cdot}_{\var}$ denotes the variation semi-norm,
    defined for any vector by
    \begin{equation} \label{eq:intro:def_var}
        \forall N, \forall h \in \RR^N,~
        \norm{h}_{\var}
        = \frac12 (\max h - \min h) 
        = \inf_{b \in \RR} \norm{h-b \bmone_N}_\infty.
    \end{equation}
\end{lemma}

\begin{proof}
    Let $g, \tg \in \RR^n$. By definition, for any $i$ and $b \in \RR$,
    \begin{align}
        f[g]_i - f[\tg]_i
        + b
        &= \tau \log \frac{
            \sum_j e^{[-C_{ij}+g_j+b]/\tau} \nu_j
        }{
            \sum_{j'} e^{[-C_{ij'}+\tg_{j'}]/\tau} \nu_{j'}
        }
        = \tau \log \frac{
            \sum_j e^{[-C_{ij}+\tg_j]/\tau}
            \,
            e^{[g_j-\tg_j+b]/\tau}
            \nu_j
        }{
            \sum_{j'} e^{[-C_{ij'}+\tg_{j'}]/\tau} \nu_{j'}
        } \\
        &\leq \tau \log \frac{
            \sum_j e^{[-C_{ij}+\tg_j]/\tau}
            \,
            e^{\norm{g-\tg+b \bmone_n}_\infty/\tau}
            \nu_j
        }{
            \sum_{j'} e^{[-C_{ij'}+\tg_{j'}]/\tau} \nu_{j'}
        } 
        = \norm{g-\tg+b \bmone_n}_\infty
    \end{align}
    and symmetrically, in the other direction,
    \begin{equation}
        f[\tg]_i - f[g]_i
        - b
        \leq 
        \norm{\tg-g-b \bmone_n}_\infty.
    \end{equation}
    Thus,
    \begin{align*}
        \forall b,~~
        2 \norm{f[g]-f[\tg]}_{\var}
        &= \max \left( f[g] - f[\tg] + b \bmone_n \right)
        - \min \left( f[g] - f[\tg] + b \bmone_n \right)
        \leq 2 \norm{g-\tg+b \bmone_n}_\infty \\
        \text{and so}~~~~
        \norm{f[g]-f[\tg]}_{\var}
        &\leq \inf_{b \in \RR} \norm{g-\tg+b \bmone_n}_\infty
        = \norm{g-\tg}_{\var}.
    \end{align*}
    The statement for $g[\cdot]$ follows similarly.
\end{proof}

Finally, for ease of presentation, let us formally clarify the structure of the set of minimizers of~$\Psi$.

\begin{lemma}[Normalized minimizers] \label{lm:intro:normalized_minimizers}
    We have the equivalences
    \begin{equation}
        \forall (f^*, g^*) \in \argmin \Psi,~~
        \sum\nolimits_{ij} e^{[-C_{ij} + f^*_i + g^*_j]/\tau} \mu_i \nu_j = 1
        ~\iff~
        f^* = f[g^*]
        ~\iff~
        g^* = g[f^*].
    \end{equation}
    We will call the $(f^*, g^*)$ satisfying these conditions the \emph{normalized minimizers} of $\Psi$.
\end{lemma}

\begin{proof}
    Denote $Z(f, g) = \sum_{ij} e^{[-C_{ij}+f_i+g_j]/\tau} \mu_i \nu_j$.
    The implications 
    $f^* = f[g^*]$ or $g^* = g[f^*] \implies Z(f^*, g^*) = 1$
    can be checked by explicit computations (and they hold for any vectors $f^*, g^*$, not just for minimizers).
    For the other direction, one can check by explicit computations that
    $\forall g, \argmin \Psi(\cdot, g) = \{ f[g] + b \bmone_m, b \in \RR \}$
    and
    $\forall f, \argmin \Psi(f, \cdot) = \{ g[f] + b \bmone_n, b \in \RR \}$.
    So for any $(f^*, g^*) \in \argmin \Psi$,
    since $f^* \in \argmin \Psi(\cdot, g^*)$,
    there exists $b \in \RR$ such that $f^* = f[g^*] + b \bmone_m$, and
    \begin{equation}
        Z(f^*, g^*)
        = Z(f[g^*] + b \bmone_m, g^*)
        = e^b\, Z(f[g^*], g^*)
        = e^b
    \end{equation}
    by definition of $Z(\cdot, \cdot)$.
    Thus, if $Z(f^*, g^*) = 1$ then $b=0$ and $f^* = f[g^*]$.
\end{proof}

\begin{remark}
    One can show that the normalized minimizers of $\Psi$ are precisely the minimizers of $\tPsi(f, g) = \tau \sum_{ij} e^{[-C_{ij}+f_i+g_j]/\tau} \mu_i \nu_j - \tau - \mu^\top f - \nu^\top g$.
    Moreover, for any normalized minimizer $(f^*, g^*)$, 
    $\argmin \Psi = \{ (f^* + b \bmone_m, g^* + b' \bmone_n), b, b' \in \RR \}$,
    while 
    ${ \argmin \tPsi = \{ (f^* + b \bmone_m, g^* - b \bmone_n), b \in \RR \} }$.
\end{remark}

\begin{remark}
    The slow polynomial convergence bound $E_k \leq O(1/\sqrt{k})$ follows immediately from \autoref{lm:intro:Psik-Psik+1} and \autoref{lm:intro:sqgrad_decr}.
    Indeed, denoting $V_k = \Psi(f^k, g^k) - \inf \Psi$ and $G_k^2 = \Hdiv{\mu}{X_\sharp \pi^k} + \Hdiv{\nu}{Y_\sharp \pi^k}$, by taking a telescopic sum in \autoref{lm:intro:Psik-Psik+1} we have
    \begin{equation}
        \frac12 E_k^2 \leq G_k^2 \leq \frac{1}{k} \sum_{s=1}^k G_s^2
        = \frac{1}{\tau k} \sum_{s=1}^k (V_s - V_{s+1})
        = \frac{1}{\tau k} (V_1 - V_{k+1})
        \leq \frac{V_1}{\tau k}
    \end{equation}
    where the first inequality follows from Pinsker's inequality
    (and the fact that at any iteration, either $\norm{X_\sharp \pi^k - \mu}_1 = 0$ or $\norm{Y_\sharp \pi^k - \nu}_1 = 0$ by \autoref{lm:intro:marginals_match}) 
    and the second inequality follows from \autoref{lm:intro:sqgrad_decr}.
    Note that $V_1 = \Psi(f[g^0], g^0) - \inf \Psi$ is finite as soon as $\inf \Psi>-\infty$, i.e., as soon as $A$ is asymptotically scalable.
    All of the above ideas are contained in \cite{altschuler2017near} and \cite{leger2021gradient}.
\end{remark}

\section{Warm-up: fast polynomial bound in the exactly scalable case} \label{sec:warmup}

In this section, we suppose $A$ is exactly scalable, i.e., $\Psi$ attains its infimum at some $(f^*, g^*)$.
As a warm-up, we present a complete convergence analysis of the Sinkhorn algorithm for this case.
We proceed by following the same ideas as \cite{dvurechensky2018computational} for the case where $A$ is positive, and examining the steps where adaptations are needed.

Our result is as follows. The remainder of this section is dedicated to its proof.
\begin{theorem} \label{thm:warmup:fastexact}
    Suppose that $A$ is exactly scalable.
    Let $K, \theta$ be defined by
    \begin{equation} \label{eq:warmup:def_K_theta}
        K = \max_{(i,j) \in \EEE} C_{ij}
        - \tau \log \left( \mu_{\min} \vee \nu_{\min} \right)
        \qquad \text{and} \qquad
        \theta = -\tau \log \sum_{ij} e^{-C_{ij}/\tau} \mu_i \nu_j
    \end{equation}
    and $\Delta>0$ defined by
    \begin{equation} \label{eq:warmup:def_Delta}
        \Delta = \min_{\substack{I \subset \{1 \dots m\} \\ J \subset \{1 \dots n\}}}~
        \abs{\sum_{i \in I} \mu_i - \sum_{j \in J} \nu_j}
        ~~~~\text{subject to}~~~~
        \sum_{i \in I} \mu_i \neq \sum_{j \in J} \nu_j.
    \end{equation}
    Further suppose the Sinkhorn algorithm is initialized with $g^0 = 0$.
    Then the $\ell_1$-norm marginal error
    $E_k = \norm{X_\sharp \pi^k - \mu}_1 + \norm{Y_\sharp \pi^k - \nu}_1$
    is bounded as 
    \begin{equation}
        \forall k \geq 3,~
        E_k \leq \frac{4 \sqrt{2}\, (K-\theta)}{\tau \Delta} \frac{1}{\sqrt{k (k-2)}}.
    \end{equation}
\end{theorem}

\begin{remark} \label{rk:warmup:K_theta}
    The quantities $K, \theta$ are ordered as $K \geq \theta$.
    Indeed,
    \begin{equation}
        \theta = -\tau \log \sum_{(i,j) \in \EEE} e^{-C_{ij}/\tau} \mu_i \nu_j
        \leq -\tau \log \min_{(i,j) \in \EEE} e^{-C_{ij}/\tau} \cdot \sum_{(i,j) \in \EEE} \mu_i \nu_j
        = \max_{(i,j) \in \EEE} C_{ij}
        - \tau \log \sum_{(i,j) \in \EEE} \mu_i \nu_j
    \end{equation}
    and $\sum_{(i,j) \in \EEE} \mu_i \nu_j
    = \sum_i \sum_{j: (i,j) \in \EEE} \mu_i \nu_j
    \geq \sum_i \mu_i \nu_{\min} = \nu_{\min}$
    and likewise $\sum_{(i,j) \in \EEE} \mu_i \nu_j \geq \mu_{\min}$, so
    $\theta \leq \max_\EEE C - \tau \log (\mu_{\min} \vee \nu_{\min}) = K$.
    
    Moreover, tracking the equality cases shows that
    $\theta = K$ if and only if $C_{ij} = C_{i'j'}$ for all $(i,j), (i',j') \in \EEE$, $m=n$, $\EEE = \{ (i, \sigma(i)), i \leq m \}$ for some permutation $\sigma$, and $\mu = \nu = (\frac1m, ..., \frac1m)$.
\end{remark}

\begin{remark} \label{rk:warmup:ultheta}
    As $\tau \to 0$, by a classical property of the log-sum-exp function, $\theta$ converges to
    \begin{equation} \label{eq:warmup:def_ultheta}
        \ul\theta = \min_{ij} C_{ij}.
    \end{equation}
    One can also check that $\theta \geq \ul\theta$ for any $\tau>0$, so $\theta$ can be replaced by $\ul\theta$ in the theorem statement for a simpler (but looser) upper bound.
\end{remark}

\begin{remark}
    The quantity $\Delta$ arises naturally from adapting the arguments of \cite{kalantari2008complexity}.
    For ease of comparison, let us mention that a different convention for the target marginals' scaling was adopted there,
    and that the corresponding quantity of interest is the one denoted by ``$h$'' in that reference.
\end{remark}

\subsection{Main body of the analysis}

The main body of the analysis actually follows along exactly the same lines as \cite{dvurechensky2018computational} for the case where $A$ is positive; no adaptations are needed.
Difficulties will arise only for bounding the quantity $D$ appearing in \eqref{eq:warmup:unif_bd_var} in \autoref{lm:warmup:Psik_leq_1/k}, which is done in the next subsection.

\begin{lemma} \label{lm:warmup:convexity_ineq}
    For any minimizer $(f^*, g^*)$ of $\Psi$,
    for any $f, g$,
    \begin{align}
        \Psi(f, g) - \Psi(f^*, g^*)
        &\leq 
        \sum_i (f_i-f^*_i) ((X_\sharp \pi^k)_i - \mu_i)
        + \sum_j (g_j-g^*_j) ((Y_\sharp \pi^k)_j - \nu_j) \\
        &\leq \left( \norm{f-f^*}_{\var} \vee \norm{g-g^*}_{\var} \right)
        E_k
    \end{align}
    where we recall that $\norm{\cdot}_{\var}$ denotes the variation semi-norm defined in \eqref{eq:intro:def_var}
    and
    $E_k = \norm{X_\sharp \pi^k - \mu}_1 + \norm{Y_\sharp \pi^k - \nu}_1$.
\end{lemma}

\begin{proof}
    By convexity of $\Psi$,
    \begin{equation}
        \Psi(f,g) - \Psi(f^*, g^*)
        \leq \begin{pmatrix}
            f-f^* \\ g-g^*
        \end{pmatrix}^\top
        \begin{pmatrix}
            \nabla_f \Psi(f, g) \\
            \nabla_g \Psi(f, g)
        \end{pmatrix},
    \end{equation}
    hence the first inequality by substituting the $\nabla_f \Psi, \nabla_g \Psi$ by their values computed in \autoref{lm:intro:nablaPsi}.
    The second inequality follows by noting that
    for any $h \in \RR^N$ and $\mu, \mu' \in \Delta_N$, by H\"older's inequality,
    \begin{equation}
        \forall b \in \RR,~
        \sum_i h_i (\mu_i-\mu'_i)
        = \sum_i (h_i - b) (\mu_i-\mu'_i)
        \leq \norm{h-b \bmone_N}_\infty \norm{\mu-\mu'}_1
    \end{equation}
    and so by taking an infimum over $b$, 
    $\sum_i h_i (\mu_i - \mu'_i) \leq \norm{h}_{\var} \norm{\mu-\mu'}_1$.
\end{proof}

\begin{lemma} \label{lm:warmup:Psik_leq_1/k}
    Suppose that 
    \begin{equation} \label{eq:warmup:unif_bd_var}
        \sup_{k \geq k_0}
        \norm{f^k-f^*}_{\var} \vee \norm{g^k-g^*}_{\var}
        \leq D
    \end{equation}
    for some minimizer $(f^*, g^*)$ of $\Psi$, some $k_0 \geq 1$, and some $D < \infty$.
    Then
    \begin{align}
        \forall k \geq k_0,
        &~~ 
        \left[ \Psi(f^k, g^k) - \Psi(f^*, g^*) \right]^2 \leq \frac{2 D^2}{\tau} \left( \Psi(f^k, g^k) - \Psi(f^{k+1}, g^{k+1}) \right) \\
        \text{and}~~~~
        \forall k \geq k_0+1,
        &~~~ 
        \Psi(f^k, g^k) - \Psi(f^*, g^*)
        \leq \frac{2 D^2}{\tau} \frac{1}{k-k_0}.
    \end{align}
\end{lemma}

\begin{proof}
    Let $k \geq k_0$. By applying the previous lemma to $f^k, g^k$ and squaring both sides of the inequality, we have
    \begin{align}
        \left[ \Psi(f^k, g^k) - \Psi(f^*, g^*) \right]^2 
        &\leq D^2 \left( \norm{X_\sharp \pi^k - \mu}_1 + \norm{Y_\sharp \pi^k - \nu}_1 \right)^2 \\
        &= D^2 \left( \norm{X_\sharp \pi^k - \mu}_1^2 + \norm{Y_\sharp \pi^k - \nu}_1^2 \right) 
    \end{align}
    since $X_\sharp \pi^k - \mu = 0$ for $k$ odd and $Y_\sharp \pi^k - \nu = 0$ for $k$ even, by \autoref{lm:intro:marginals_match}.
    So by Pinsker's inequality and \autoref{lm:intro:Psik-Psik+1},
    \begin{align}
        \left[ \Psi(f^k, g^k) - \Psi(f^*, g^*) \right]^2 
        &\leq 2 D^2 \left( \Hdiv{\mu}{X_\sharp \pi^k} + \Hdiv{\nu}{Y_\sharp \pi^k} \right) \\
        &= \frac{2 D^2}{\tau} \left( \Psi(f^k, g^k) - \Psi(f^{k+1}, g^{k+1}) \right).
    \end{align}
    
    To deduce the second inequality of the lemma, apply \autoref{lm:warmup:Vk_Ak} below to $V_k = \Psi(f^k, g^k) - \Psi(f^*, g^*)$.
\end{proof}

\begin{lemma} \label{lm:warmup:Vk_Ak}
    For any sequence of real numbers $(V_k)_{k \geq 0}$ such that
    $\forall k, V_{k+1} \leq V_k - A V_k^2$ for some constant $A > 0$, it holds
    $\forall k \geq 1, V_k \leq \frac{1}{A k}$.
\end{lemma}

\begin{proof}
    Note that $(V_k)_{k \geq 0}$ is non-increasing since $\forall k, V_{k+1} \leq V_k - 0$.
    First suppose that $V_k>0$ for all $k$.
    Dividing both sides of the inequality 
    $V_{k+1} \leq V_k - A V_k^2$
    by $V_k V_{k+1}$, we have
    \begin{equation}
        \frac{1}{V_k} 
        \leq \frac{1}{V_{k+1}} - A \frac{V_k}{V_{k+1}}
        \leq \frac{1}{V_{k+1}} - A.
    \end{equation}
    So by a telescopic sum,
    $\frac{1}{V_0} \leq \frac{1}{V_k} - k A$ and so
    $V_k \leq \frac{1}{V_0^{-1} + kA} \leq \frac{1}{kA}$.

    Now suppose there exists $k$ such that $V_k \leq 0$ and denote by $K$ the smallest such integer.
    Applying the above reasoning for all $k < K$ shows that $\forall k < K, V_k \leq \frac{1}{kA}$, and for all later indices we have of course $\forall k \geq K, V_k \leq V_K \leq 0 \leq \frac{1}{A k}$ by monotonicity.
\end{proof}

It remains to translate this $O(1/k)$ bound on $\Psi(f^k, g^k) - \Psi(f^*, g^*)$ into a $O(1/k)$ bound on $E_k$.
This can be done using a ``switching'' strategy presented in \cite{dvurechensky2018computational}, or using a doubling trick which first appeared explicitly in the literature on computational optimal transport in \cite[Proposition~4.3]{ghosal2025convergence}.
We follow the latter path.

\begin{lemma} \label{lm:warmup:Ek_leq_1/k}
    For any $k \geq 2$, we have
    \begin{equation}
        E_k^2 \leq \frac{8}{\tau k} \left( \Psi(f^{\ceil{k/2}}, g^{\ceil{k/2}}) - \min \Psi \right).
    \end{equation}
    Consequently, if \eqref{eq:warmup:unif_bd_var} holds with some 
    $(f^*, g^*) \in \argmin \Psi$,
    $k_0 \geq 1$, 
    and $D<\infty$, 
    then
    \begin{equation}
        \forall k > 2 k_0,~
        E_k \leq \frac{4 \sqrt{2}\, D}{\tau} \frac{1}{\sqrt{k (k-2k_0)}}.
    \end{equation}
\end{lemma}

\begin{proof}
    Denote for concision 
    \begin{equation}
        \forall k,~~
        V_k = \Psi(f^k, g^k) - \min \Psi
        \qquad\text{and}\qquad
        G_k^2 = \Hdiv{\mu}{X_\sharp \pi^k} + \Hdiv{\nu}{Y_\sharp \pi^k}.
    \end{equation}
    We know by \autoref{lm:intro:Psik-Psik+1} that $V_k - V_{k+1} = \tau G_k^2$ and by \autoref{lm:intro:sqgrad_decr} that $(G_{2k})_k, (G_{2k+1})_k$ are non-increasing.
    So for any $k \geq 1$ even,
    \begin{align}
        & \frac2k \sum_{s=k}^{2k-1} G_s^2 
        = \frac2k \sum_{\substack{s=k\\ s \text{ even}}}^{2k-1} G_s^2
        + \frac2k \sum_{\substack{s=k\\ s \text{ odd}}}^{2k-1} G_s^2
        \geq G_{2k-2}^2 + G_{2k-1}^2 \\
        \text{and}\qquad
        & \frac2k \sum_{s=k}^{2k-1} G_s^2 
        = \frac2k \sum_{s=k}^{2k-1} \frac1\tau (V_s - V_{s+1})
        = \frac{2}{\tau k} (V_k - V_{2k})
        \leq \frac{2 V_k}{\tau k}.
    \end{align}
    Hence for any $k \geq 2$, by Pinsker's inequality,
    \begin{equation}
        E_k^2 \leq 2\, G_k^2
        \leq \frac{4 V_{\ceil{k/2}}}{\tau \ceil{k/2}}
        \leq \frac{8 V_{\ceil{k/2}}}{\tau k}.
    \end{equation}
    This proves the first announced inequality.
    Combined with \autoref{lm:warmup:Psik_leq_1/k},
    we get that for any $k > 2 k_0$,
    \begin{equation}
        E_k^2 \leq \frac{8}{\tau k} \frac{2 D^2}{\tau} \frac{1}{(\ceil{k/2}-k_0)}
        \leq \frac{16 D^2}{\tau^2} \frac{1}{k (k/2-k_0)}
        \leq \frac{32 D^2}{\tau^2} \frac{1}{k (k-2k_0)},
    \end{equation}
    hence the second announced inequality by taking square roots on both sides.
\end{proof}

It still remains to identify a pair $(f^*, g^*) \in \argmin \Psi$ and constants $k_0 \geq 1$ and $D < \infty$ such that \eqref{eq:warmup:unif_bd_var} holds.
This is the object of the next subsection.

\begin{remark}
    In the case where $A$ is positive, i.e., if $\forall i,j, C_{ij}<\infty$, then one can show by explicit computations that
    $\max_f \norm{g[f]}_{\var} \leq \max_i \norm{C_{i\bullet}}_{\var}$, and symmetrically for $f[\cdot]$.
    This immediately gives a usable bound to plug into \eqref{eq:warmup:unif_bd_var}.
    A slightly tighter bound can also be obtained by using that
    \begin{equation}
        \max_{f,f'} \norm{g[f]-g[f']}_{\var} 
        = \max_{g,g'} \norm{f[g]-f[g']}_{\var} 
        = \frac{1}{2\tau}\, \max_{i, i', j, j'}\, C_{ij} - C_{ij'} - C_{i'j} + C_{i'j'}
    \end{equation}
    as shown, e.g., in \cite[Theorem~6.2]{eveson1995elementary}.
    Of course, these bounds are available only if $A$ is positive.
\end{remark}

\subsection{Uniform bound on the variation semi-norm of iterates}

Recall that we call normalized minimizer of $\Psi$ any $(f^*, g^*) \in \argmin \Psi$ satisfying the equivalent conditions of \autoref{lm:intro:normalized_minimizers}.

\begin{lemma} \label{lm:warmup:easy_D}
    For any normalized minimizer $(f^*, g^*)$ of $\Psi$, \eqref{eq:warmup:unif_bd_var} holds with $k_0 = 1$ and
    $D = \norm{g^*-g^0}_{\var}$.
\end{lemma}

\begin{proof}
    Consider the sequence $(u_k)_{k \geq 0}$ defined by
    $u_k = \begin{cases}
        \norm{f^k-f^*}_{\var} ~\text{if $k$ is odd} \\
        \norm{g^k-g^*}_{\var} ~\text{if $k$ is even}
    \end{cases}$.
    Since $f^* = f[g^*]$ and $g^* = g[f^*]$, then by \autoref{lm:intro:varseminorm_nonexpansive}, this sequence is non-increasing.
    Hence, \eqref{eq:warmup:unif_bd_var}
    holds with $k_0 = 1$ and $D = u_0 = \norm{g^0-g^*}_{\var}$.
\end{proof}

If the Sinkhorn algorithm is initialized with $g^0=0$, which is the standard choice in practice, it only remains to bound $\inf_{(f^*, g^*) \in \argmin \Psi} \norm{g^*}_{\var}$.
This is done in the following proposition, which is adapted from \cite[Theorem~5.1]{kalantari2008complexity}.
It also shows a bound on the supremum norms which will be used in \autoref{sec:fastasymp}.

\begin{proposition} \label{prop:warmup:good_minimizer}
    There exists a normalized minimizer $(f^*, g^*)$ of $\Psi$ such that
    \begin{equation}
        \norm{f^*}_{\var}, \norm{g^*}_{\var}
        \leq \frac{K-\theta}{\Delta}
        \qquad\text{and}\qquad
        \norm{f^*}_\infty, \norm{g^*}_\infty \leq \frac{K}{2} + \frac{K-\theta}{\Delta},
    \end{equation}
    where $K, \theta, \Delta$ are the quantities defined in \eqref{eq:warmup:def_K_theta} and \eqref{eq:warmup:def_Delta}.
\end{proposition}

This proposition follows from exactly the same arguments as \cite[Theorem~5.1]{kalantari2008complexity}. 
However a different convention for the target marginals' scaling was adopted in that reference (see its Eqs.~(3), (4)) and carrying out the necessary adaptations can be tedious.
So the proof of this proposition is presented in full in \autoref{sec:apx_delayedpf_warmup} for the reader's convenience.

We can now conclude the proof of \autoref{thm:warmup:fastexact}.

\begin{proof}[Proof of \autoref{thm:warmup:fastexact}]
    Apply \autoref{lm:warmup:Ek_leq_1/k} with the minimizer $(f^*, g^*)$ exhibited in \autoref{prop:warmup:good_minimizer}, $k_0=1$, and the constant $D$ identified in \autoref{lm:warmup:easy_D}.
\end{proof}

\section{Fast polynomial bound in the asymptotically scalable case} \label{sec:fastasymp}

Suppose henceforth that $A$ is asymptotically scalable (but not necessarily exactly scalable), i.e., that $\inf \Psi > -\infty$ but $\Psi$ may not attain its infinimum at any pair of finite vectors.
Then the analysis presented in the previous section cannot be applied directly, as the ``main body of the analysis'' \autoref{lm:warmup:convexity_ineq}, \autoref{lm:warmup:Psik_leq_1/k}, and \autoref{lm:warmup:Ek_leq_1/k} assumed the existence of a minimizer of $\Psi$.
Nonetheless, the analysis can be adapted by using approximate minimizers instead.

To state our main result in its tightest form, let us recall the following structural result about the asymptotically scalable case.
Its earliest occurrence in the context of the Sinkhorn algorithm and matrix scaling, to our knowledge, is \cite[Lemma~C.3]{allen2017much},
and its connection to the Dulmage-Mendelsohn decomposition of bipartite graphs was pointed out in \cite{hayashi2024finding}.

\begin{proposition}[{Generalized Dulmage-Mendelsohn decomposition \cite[Lemma~C.3]{allen2017much}}] \label{prop:fastasymp:DM_decomp} 
    Suppose that $A$ is asymptotically scalable.
    Recall that $\EEE = \left\{ (i,j);~ A_{ij}>0 \right\}$.
    Then there exists an integer $P$ and partitions $\{1 \dots m \} = I_1 \sqcup ... \sqcup I_P$ and $\{1 \dots n \} = J_1 \sqcup ... \sqcup J_P$ such that
    \begin{itemize}
        \item For all $p \leq P$, 
        $\sum_{i \in I_p} \mu_i = \sum_{j \in J_p} \nu_j$
        and $A^{(p)} = (A_{ij})_{i \in I_p, j \in J_p}$ is exactly $(\restr{\mu}{I_p}, \restr{\nu}{J_p})$-scalable.
        \item Denoting by ``$\to$'' the relation on $\{1 \dots P\}$ given by
        $p \to q \iff \exists i \in I_p, j \in J_q ~\text{s.t.}~ A_{ij}>0$ and $p \neq q$,
        the directed graph $(\{1 \dots P\}, \{ (p,q);~ p \to q \} )$ is a directed acyclic graph (DAG).
        This DAG is connected provided that $\EEE$ is connected.
    \end{itemize}
\end{proposition}

The main result of this paper is as follows.
The remainder of the section is dedicated to its proof.

\begin{theorem} \label{thm:fastasymp:fastasymp}
    Suppose that $A$ is asymptotically scalable.
    Let $\ell$ denote the diameter of the DAG $(\{1 \dots P\}, \{ (p,q);~ p \to q \} )$ constructed in \autoref{prop:fastasymp:DM_decomp}, i.e., the maximal length of a path; in particular, $\ell \leq P \leq \min(m,n)$.
    Let $K, \Delta, \ul\theta$ be defined as in \eqref{eq:warmup:def_K_theta}, \eqref{eq:warmup:def_Delta}, \eqref{eq:warmup:def_ultheta}; that is, for ease of reference,
    \begin{align}
        K &= \max_{(i,j) \in \EEE} C_{ij}
        - \tau \log \left( \mu_{\min} \vee \nu_{\min} \right),
        \qquad\qquad
        \ul\theta = \min_{ij} C_{ij}, \\
        \Delta &= \min_{\substack{I \subset \{1 \dots m\} \\ J \subset \{1 \dots n\}}}~
        \abs{\sum_{i \in I} \mu_i - \sum_{j \in J} \nu_j}
        ~~\text{subject to}~~
        \sum_{i \in I} \mu_i \neq \sum_{j \in J} \nu_j.
    \end{align}
    Further suppose the Sinkhorn algorithm is initialized with $g^0 = 0$.
    Then the $\ell_1$-norm marginal error
    $E_k = \norm{X_\sharp \pi^k - \mu}_1 + \norm{Y_\sharp \pi^k - \nu}_1$
    and the dual suboptimality $V_k = \Psi(f^k, g^k) - \inf \Psi$
    are bounded as 
    \begin{equation}
        \forall k \geq 2 e \ell^2 + 3,~~
        E_k 
        \leq \sqrt{\frac{8 V_{\ceil{k/2}}}{\tau k}}
        \leq \frac{4 \sqrt{2}}{\sqrt{k (k-2)}} \left[ 
            \frac{K-\ul\theta}{\tau} \left( \ell + \frac{2 (\ell+1)}{\Delta} \right)
            + \ell\, \log \frac{k-2}{2 \ell^2}
        \right].
    \end{equation}
\end{theorem}


\subsection{Reduction to estimating the rate function}

Define the \emph{rate function}
\cite{chizat2022convergence}
\begin{equation} \label{eq:fastasymp:def_Q}
    \forall \alpha \geq 0,~~
    Q(\alpha) = \inf_{\hat g \in \RR^n} \left[ \Psi(f[\hat g], \hat g) - \inf \Psi + \alpha \norm{\hat g-g^0}_{\var}^2 \right]
\end{equation}
where we recall that $\norm{\cdot}_{\var}$ denotes the variation semi-norm defined in \eqref{eq:intro:def_var}.

\begin{lemma} \label{lm:fastasymp:Psik_Ek_leq_Q}
    For any $\hat g \in \RR^n$,
    \begin{equation}
        \forall k \geq 2,~
        \Psi(f^k, g^k) - \Psi(f[\hat g], \hat g)
        \leq \frac{2 \norm{\hat g - g^0}_{\var}^2}{\tau (k-1)}.
    \end{equation}
    As a consequence,
    \begin{equation}
        \forall k \geq 2,~
        \Psi(f^k, g^k) - \inf \Psi
        \leq Q(\alpha_k)
        \qquad \text{where} \qquad
        \alpha_k = \frac{2}{\tau (k-1)}.
    \end{equation}
    Moreover,
    \begin{equation}
        \forall k \geq 3,~~
        E_k^2 
        \leq \frac{8}{\tau k} \left( \Psi(f^{\ceil{k/2}}, g^{\ceil{k/2}}) - \inf \Psi \right)
        \leq \frac{8}{\tau k} \,
        Q(\alpha_{\ceil{k/2}}).
    \end{equation}
\end{lemma}

\begin{proof}
    Denote by $(\hat f^k, \hat g^k)_{k \geq 0}$ the iterates of the Sinkhorn algorithm initialized with $\hat g^0 = \hat g$.
    By \autoref{lm:intro:varseminorm_nonexpansive},
    the sequence 
    $u_k = \begin{cases}
        \norm{f^k-\hat f^k}_{\var} ~\text{if $k$ is odd} \\
        \norm{g^k-\hat g^k}_{\var} ~\text{if $k$ is even}
    \end{cases}$
    is non-increasing, so
    \begin{equation}
        \sup_{k \geq 1}
        \norm{f^k-\hat f^k}_{\var} \vee \norm{g^k-\hat g^k}_{\var}~
        \leq u_0
        = \norm{\hat g - g^0}_{\var}.
    \end{equation}
    
    Now observe that in the proof of \autoref{lm:warmup:convexity_ineq},
    we never actually used the assumption that $(f^*, g^*)$ was a minimizer of $\Psi$.
    So by applying the same arguments with $(f^*, g^*)$ replaced by $(\hat f^k, \hat g^k)$, we get
    \begin{equation}
        \forall k \geq 1,~~
        \Psi(f^k, g^k) - \Psi(\hat f^k, \hat g^k)
        \leq \left( \norm{f^k-\hat f^k}_{\var} \vee \norm{g^k - \hat g^k}_{\var} \right)
        E_k
        \leq \norm{\hat g - g^0}_{\var} E_k.
    \end{equation}
    Then, by applying the same arguments as in the proof of the first part of \autoref{lm:warmup:Psik_leq_1/k}, we have
    \begin{equation}
        \left[ \Psi(f^k, g^k) - \Psi(f[\hat g], \hat g) \right]^2 
        \leq \left[ \Psi(f^k, g^k) - \Psi(\hat f^k, \hat g^k) \right]^2
        \leq \frac{2 \norm{\hat g - g^0}_{\var}^2}{\tau} \left( \Psi(f^k, g^k) - \Psi(f^{k+1}, g^{k+1}) \right)
    \end{equation}
    where for the first inequality we additionally used that $(\Psi(\hat f^k, \hat g^k))_{k \geq 0}$ is non-increasing, by definition of the Sinkhorn algorithm as an alternating minimization scheme.
    So by applying \autoref{lm:warmup:Vk_Ak} to $V_k = \Psi(f^k, g^k) - \Psi(f[\hat g], \hat g)$, we obtain that
    \begin{equation}
        \forall k \geq 2,~
        \Psi(f^k, g^k) - \Psi(f[\hat g], \hat g)
        \leq \frac{2 \norm{\hat g - g^0}_{\var}^2}{\tau} \frac{1}{k-1},
    \end{equation}
    as desired.

    To deduce the second inequality of the lemma, simply rewrite the first inequality of the lemma~as
    \begin{equation}
        \forall \hat g, \forall k \geq 2,~~
        \Psi(f^k, g^k) - \inf \Psi
        \leq \Psi(f[\hat g], \hat g) - \inf \Psi + \frac{2 \norm{\hat g - g^0}_{\var}^2}{\tau (k-1)}
    \end{equation}
    and take the infimum over $\hat g$ on the right-hand side, $k$ being fixed.

    The third inequality of the lemma follows by the first part of \autoref{lm:warmup:Ek_leq_1/k}, since one can check that its proof did not use the existence of a minimizer of $\Psi$.
\end{proof}

\subsection{Existence of good approximate minimizers and rate function estimate}

Thanks to \autoref{lm:fastasymp:Psik_Ek_leq_Q}, to obtain a convergence upper bound for the algorithm, it suffices to control the growth of the rate function
$Q(\alpha) = \inf_{\hat g \in \RR^n} \left[ \Psi(f[\hat g], \hat g) - \inf \Psi + \alpha \norm{\hat g-g^0}_{\var}^2 \right]$
for $\alpha$ small.
In other words, we wish to show the existence of approximate minimizers $\hat g$ that simultaneously have a low variation semi-norm (assuming $g^0=0$).
This is done in the next proposition, which is an adapted and refined version of \cite[Lemma~3.3]{allen2017much}.

\begin{proposition} \label{prop:fastasymp:good_approx_minimizer}
    For any $0 < \eps \leq \tau\, e^{\tau^{-1} (K-\ul\theta) (1+2/\Delta)}$, there exists $\hat g_\eps \in \RR^m$ such that
    \begin{equation}
        \Psi(f[\hat g_\eps], \hat g_\eps) - \inf \Psi \leq \eps
        \qquad \text{and} \qquad 
        \norm{\hat g_\eps}_{\var}
        \leq
        \frac{\tau \ell}{2} \log (\tau/\eps)
        + (K-\ul\theta)
        \left( \frac{\ell}{2} + \frac{1 + \ell}{\Delta} \right)
    \end{equation}
    where $\ell, K, \Delta, \ul\theta$ are the quantities defined in the statement of \autoref{thm:fastasymp:fastasymp}.
\end{proposition}

The proof of this proposition is technical and is delayed to \autoref{sec:apx_delayedpf_GAM}.

The above existence result translates to the following estimate on the growth of the rate function.

\begin{lemma} \label{lm:fastasymp:ub_Q}
    Suppose $g^0=0$. 
    Let $M = (K-\ul\theta) \left( \frac{\ell}{2} + \frac{1 + \ell}{\Delta} \right)$.
    Then the rate function $Q(\alpha)$ defined in \eqref{eq:fastasymp:def_Q} satisfies
    \begin{equation}
        \forall 0< \alpha \leq
        \frac{
            e^{\tau^{-1} (K-\ul\theta) (1+2/\Delta)}
        }{
            e \, \tau \ell \, \max\left\{
                \ell/2, 
                \frac{K-\ul\theta}{\tau \Delta}
            \right\}
        }, \quad
        Q(\alpha) \leq \alpha \tau^2
        \left[ 
            \frac{2 M}{\tau}
            + \ell\, \log \frac{2}{\alpha \tau \ell^2}
        \right]^2.
    \end{equation}
\end{lemma}

\begin{proof}
    By plugging the good approximate minimizers identified in \autoref{prop:fastasymp:good_approx_minimizer} into the infimum defining $Q(\alpha)$, we have
    \begin{align}
        Q(\alpha)
        &\leq \inf_{0<\eps \leq \ol\eps}
        \left[ \Psi(f[\hat g_\eps], \hat g_\eps) - \inf \Psi + \alpha \norm{\hat g_\eps}_{\var}^2 \right]
        \leq \inf_{0<\eps \leq \ol\eps}\,
        \eps + \alpha \left(
            \frac{\tau \ell}{2} \log(\tau/\eps) + M
        \right)^2 \\
        \tau^{-1} Q(\tau \talpha)
        &\leq 
        \inf_{0<\eps \leq \ol\eps}\,
        \eps/\tau + \talpha \left(
            \frac{\tau \ell}{2} \log(\tau/\eps) + M
        \right)^2
        =
        \inf_{0< x \leq \ol\eps/\tau}\,
        x + \talpha \left(
            \frac{\tau \ell}{2} \log(1/x) + M
        \right)^2
    \end{align}
    where $\ol\eps = \tau\, e^{\tau^{-1} (K-\ul\theta) (1+2/\Delta)}$ and $\talpha = \tau^{-1} \alpha$.
    So let us apply \autoref{lm:fastasymp:technical_ub_inf} below with
    $b = \frac{\tau \ell}{2}$,
    $X = \ol\eps/\tau$.
    The condition $X < e^{M/b}$ is indeed satisfied since
    \begin{align*}
        \log\left( e^{M/b} / X \right)
        = \frac{M}{b} - \log X
        = \frac{2 M}{\tau \ell} - \log (\ol\eps/\tau) 
        &= \frac{2}{\tau \ell} (K-\ul\theta) \left( \frac\ell2 + \frac{1+\ell}{\Delta} \right)
        - \tau^{-1} (K-\ul\theta) \left( 1 + \frac2\Delta \right) \\
        &= \tau^{-1} (K-\ul\theta)
        \left[ 
            1 + \frac{2 (1+\ell)}{\ell \Delta}
            - \left( 1 + \frac2\Delta \right)
        \right] \\
        &= \tau^{-1} (K-\ul\theta)\,
        \frac{2}{\ell \Delta}
        > 0.
    \end{align*}
    Applying the lemma, we obtain
    \begin{align}
        \tau^{-1} Q(\tau \talpha)
        &\leq 4 \talpha b^2
        \left[
            \log\left( \frac{e^{M/b}}{2 \talpha b^2} \right)
        \right]^2
        = \talpha \tau^2 \ell^2
        \left[ 
            \frac{2 M}{\tau \ell}
            + \log \frac{2}{\talpha \tau^2 \ell^2}
        \right]^2 
        = \talpha \tau^2
        \left[ 
            \frac{2 M}{\tau}
            + \ell\, \log \frac{2}{\talpha \tau^2 \ell^2}
        \right]^2 \\
        Q(\alpha)
        &\leq \alpha \tau^2
        \left[ 
            \frac{2 M}{\tau}
            + \ell\, \log \frac{2}{\alpha \tau \ell^2}
        \right]^2,
    \end{align}
    provided that
    \begin{equation}
        \tau^{-1} \alpha
        = \talpha
        \leq
        \frac{X}{2 e \, b^2 \, \max\left\{1, \log(e^{M/b}/X) \right\}}
        = \frac{
            2\, e^{\tau^{-1} (K-\ul\theta) (1+2/\Delta)}
        }{
            e \, \tau^2 \ell^2 \, \max\left\{
                1, 
                \frac{2}{\ell} \frac{K-\ul\theta}{\tau \Delta}
            \right\}
        }
        = \frac{
            e^{\tau^{-1} (K-\ul\theta) (1+2/\Delta)}
        }{
            e \, \tau^2 \ell \, \max\left\{
                \ell/2, 
                \frac{K-\ul\theta}{\tau \Delta}
            \right\}
        },
    \end{equation}
    as announced.
\end{proof}

The proof of the following technical lemma is delayed to \autoref{sec:apx_delayedpf_technical}.

\begin{lemma} \label{lm:fastasymp:technical_ub_inf}
    Let $b, X>0, M \in \RR$ such that $X < e^{M/b}$. 
    Then for all
    $0 < \alpha \leq
    \frac{X}{2 e b^2 \max\left\{ 1, \log(e^{M/b}/X) \right\}}$,
    \begin{equation}
        \inf_{0<x\le X}
        x + \alpha \Big( b \log(1/x) + M \Big)^2
        \leq 4 \alpha b^2
        \left[
            \log\left( \frac{e^{M/b}}{2\alpha b^2} \right)
        \right]^2.
    \end{equation}
\end{lemma}

We can now conclude the proof of \autoref{thm:fastasymp:fastasymp}.

\begin{proof}[Proof of \autoref{thm:fastasymp:fastasymp}]
    Let $\alpha_k = \frac{2}{\tau (k-1)}$,
    $\ol\alpha = \frac{
        e^{\tau^{-1} (K-\ul\theta) (1+2/\Delta)}
    }{
        e \, \tau \ell \, \max\left\{
            \ell/2, 
            \frac{K-\ul\theta}{\tau \Delta}
        \right\}
    }$,
    and $M = (K-\ul\theta) \left( \frac{\ell}{2} + \frac{1 + \ell}{\Delta} \right)$.
    For any $k \geq \frac{2}{\tau \ol\alpha}+1$, we have $\alpha_k \leq \ol\alpha$ and so by \autoref{lm:fastasymp:ub_Q},
    \begin{equation}
        Q(\alpha_k)
        \leq \alpha_k \tau^2
        \left[ 
            \frac{2 M}{\tau}
            + \ell\, \log \frac{2}{\alpha_k \tau \ell^2}
        \right]^2
        = \frac{2 \tau}{k-1}
        \left[
            \frac{2 M}{\tau}
            + \ell\, \log \frac{k-1}{\ell^2}
        \right]^2.
    \end{equation}
    So for any $k \geq 3 \vee (\frac{4}{\tau \ol\alpha} + 2)$ even, by \autoref{lm:fastasymp:Psik_Ek_leq_Q},
    \begin{align}
        E_k^2 
        \leq \frac{8}{\tau k} \left( \Psi(f^{k/2}, g^{k/2}) - \inf \Psi \right)
        \leq \frac{8}{\tau k} Q(\alpha_{k/2})
        &\leq 
        \frac{8}{\tau k}\cdot
        \frac{2 \tau}{k/2-1}
        \left[
            \frac{2 M}{\tau}
            + \ell\, \log \frac{k/2-1}{\ell^2}
        \right]^2 \\
        &= 
        \frac{32}{k (k-2)}
        \left[
            \frac{2 M}{\tau}
            + \ell\, \log \frac{k-2}{2 \ell^2}
        \right]^2.
    \end{align}
    One can check using monotonicity of $Q(\cdot)$ that the same inequality holds also for $k$ odd.
    By computing
    $\frac{2M}{\tau} = \frac{K-\ul\theta}{\tau} \Big( \ell + \frac{2 (\ell+1)}{\Delta} \Big)$
    and taking square roots on both sides of the inequality, we obtain
    \begin{equation}
        \forall k \geq 3 \vee \left( \frac{4}{\tau \ol\alpha} + 2 \right),~
        E_k 
        \leq \frac{4 \sqrt{2}}{\sqrt{k (k-2)}} \left[ 
            \frac{K-\ul\theta}{\tau} \left( \ell + \frac{2 (\ell+1)}{\Delta} \right)
            + \ell\, \log \frac{k-2}{2 \ell^2}
        \right].
    \end{equation}
    It only remains to estimate the lower bound on $k$: by definition of $\ol\alpha$,
    \begin{align}
        \frac{4}{\tau \ol\alpha}
        = \frac{
            4 e \, \ell \, \max\left\{
                \ell/2, 
                \frac{K-\ul\theta}{\tau \Delta}
            \right\}
        }{
            e^{\tau^{-1} (K-\ul\theta) (1+2/\Delta)}
        }
        &\leq 4 e \, \ell 
        \max\left\{
            \ell/2,~
            \frac12 \cdot \frac{2 (K-\ul\theta)}{\tau \Delta}
            \exp\left(
                -\frac{2 (K-\ul\theta)}{\tau \Delta}
            \right)
        \right\} \\
        &\leq 4 e \, \ell 
        \max\left\{
            \ell/2,~
            \frac12
        \right\}
        = 2 e \ell^2
    \end{align}
    since $\forall x, x e^{-x} \leq 1$,
    so the above inequality holds for all 
    $k \geq 2 e \ell^2 + 3$,
    as announced.
\end{proof}

\begin{remark} \label{rk:fastasymp:soules}
    The following example, adapted from \cite[page~19]{soules1991rate}, shows that our analysis is loose in some cases and indicates that the factor $\log k$ in our convergence upper bound may be removable.
    Let $m = n = 2$, $\mu = \nu = (\frac12, \frac12)$, and 
    $A = \begin{bmatrix}
        1 & 1 \\
        0 & 1
    \end{bmatrix}$
    (or equivalently,
    $C/\tau = -\log 4 \begin{bmatrix}
        1 & 1 \\
        \infty & 1
    \end{bmatrix}$
    so that $A_{ij} = e^{-C_{ij}/\tau} \mu_i \nu_j$).
    In particular $\EEE = \left\{ (1,1), (1,2), (2,2) \right\}$.
    
    The scaling $Q(\alpha) \lesssim \alpha \log(1/\alpha)^2$ for small $\alpha$ obtained in \autoref{prop:fastasymp:good_approx_minimizer}, \autoref{lm:fastasymp:ub_Q} is sharp in this case.
    Indeed, one can show that, taking $\tau=1$ for simplicity, \cite{cuturi2018semidual}
    \begin{align}
        \forall g,~
        \Psi(f[g], g)
        &= \tau \sum_i \mu_i \log \bigg(
            \sum_j
            e^{[-C_{ij}+g_j]/\tau}
            \nu_j
        \bigg)
        - \nu^\top g \\
        &= \frac12 \log \left( 2 e^{g_1} + 2 e^{g_2} \right)
        + \frac12 \log \left( 2 e^{g_2} \right)
        - \frac12 (g_1 + g_2) \\
        &= \log 2
        + \frac12 \log \left( 1 + e^{g_2-g_1} \right).
    \end{align}
    Since $\inf_{\delta \in \RR} \log(1+e^\delta) = 0$, then $\inf \Psi = \log 2$. Thus
    \begin{equation}
        Q(\alpha) = \inf_{g \in \RR^2} \left[ \Psi(f[g], g) - \inf \Psi + \alpha \norm{g}_{\var}^2 \right]
        = \inf_{\delta \in \RR}\,
        \frac12 \log (1 + e^\delta)
        + \alpha \abs{\delta/2}^2,
    \end{equation}
    and one can show that
    $Q(\alpha) \sim \frac14 \alpha \log(1/\alpha)^2$ as $\alpha \to 0$.
    As a consequence, upon substitution into \autoref{lm:fastasymp:Psik_Ek_leq_Q},
    the best upper bound on $E_k$ one can obtain with our rate-function based approach is of order $k^{-1} \log k$.

    On the other hand, the Sinkhorn iterates (with $g^0=0$) are given by, for all $k \geq 1$,
    \begin{equation}
        \text{if $k$ is even,}~~
        \pi^{k} = \frac12 \begin{bmatrix}
            1 & \frac1k \\
            0 & 1-\frac1k
        \end{bmatrix},
        \qquad 
        \text{and if $k$ is odd,}~~
        \pi^{k} = \frac12 \begin{bmatrix}
            1-\frac1k & \frac1k \\
            0 & 1
        \end{bmatrix}.
    \end{equation}
    So for all $k$ even,
    $X_\sharp \pi^k = \frac12 \left( 1+\frac1k, 1-\frac1k \right)$ and
    $E_k = \norm{X_\sharp \pi^k - \mu}_1
    = \frac1k$.

    In view of this example, obtaining a smaller bound than $E_k \leq O(\log k/k)$ would require
    tightening the ``main body of the analysis'' presented in \autoref{lm:fastasymp:Psik_Ek_leq_Q},
    or taking a different approach than the one used in this paper.
\end{remark}

\begin{remark}
    A similar situation with a mismatch between the lower bound $\Omega(1/k)$ and the upper bound $O(\log k/k)$ obtained through a rate function, 
    even though the rate function estimate is provably sharp,
    arose in \cite[Proposition 5.5]{chizat2022convergence}.
    The setting studied there is convex optimization over the space of measures, and the aforementioned mismatch occurs for the Bregman proximal gradient method with an entropic link function (``$\eta_{\mathrm{ent}}$'' or ``$\eta_{\mathrm{hyp}}$'' in their notations).
\end{remark}

\section*{Acknowledgments}

I would like to thank Atsushi Nitanda for insightful discussions on the Sinkhorn algorithm.

A LLM (ChatGPT-5.4 Thinking) was used as an interactive research assistant throughout this project, except at the writing stage.
Notably, the construction presented in \autoref{sec:apx_delayedpf_GAM}, allowing to refine the constants in \cite[Lemma~3.3]{allen2017much},
was provided by it.
The author assumes responsibility for all content.

\printbibliography
\addcontentsline{toc}{section}{\refname} 

\ifextended%
\newpage
\appendix
\phantomsection
\addcontentsline{toc}{section}{APPENDIX}


\section{Proof of \autoref{prop:warmup:good_minimizer}} \label{sec:apx_delayedpf_warmup}

In this appendix, we present the full proof of \autoref{prop:warmup:good_minimizer}, restated below.
As explained in the main text, the proposition follows from exactly the same arguments as \cite[Theorem~5.1]{kalantari2008complexity}.
We reproduce the proof in full here only for the reader's convenience, because a different convention for the target marginals' scaling was adopted in that reference and carrying out the necessary adaptations can be tedious.

Recall that we call normalized minimizer of $\Psi$ any $(f^*, g^*) \in \argmin \Psi$ satisfying the three equivalent conditions of \autoref{lm:intro:normalized_minimizers}, i.e.,
$\sum_{ij} e^{[-C_{ij}+f^*_i+g^*_j]/\tau} \mu_i \nu_j = 1$
and $f^* = f[g^*]$ 
and $g^* = g[f^*]$.

\begin{proposition*}[\autoref{prop:warmup:good_minimizer}, restated]
    Suppose that $A$ is exactly scalable.
    Then there exists a normalized minimizer $(f^*, g^*)$ of $\Psi$ such that
    \begin{equation}
        \norm{f^*}_{\var}, \norm{g^*}_{\var}
        \leq \frac{K-\theta}{\Delta}
        \qquad\text{and}\qquad
        \norm{f^*}_\infty, \norm{g^*}_\infty \leq \frac{K}{2} + \frac{K-\theta}{\Delta}
    \end{equation}
    where $K, \theta, \Delta$ are the quantities defined in \eqref{eq:warmup:def_K_theta} and \eqref{eq:warmup:def_Delta}, that is,
    \begin{align}
        K &= \max_{(i,j) \in \EEE} C_{ij}
        - \tau \log \left( \mu_{\min} \vee \nu_{\min} \right),
        \qquad\qquad
        \theta = -\tau \log \sum_{ij} e^{-C_{ij}/\tau} \mu_i \nu_j, \\
        \Delta &= \min_{\substack{I \subset \{1 \dots m\} \\ J \subset \{1 \dots n\}}}~
        \abs{\sum_{i \in I} \mu_i - \sum_{j \in J} \nu_j}
        ~~\text{subject to}~~
        \sum_{i \in I} \mu_i \neq \sum_{j \in J} \nu_j.
    \end{align}
\end{proposition*}

The following lemma is the adaptation of \cite[Lemmas~5.1, 5.2]{kalantari2008complexity}.

\begin{lemma} \label{lm:warmup:kalantari08_lm5152}
    Let $\Omega$ be the set of $(f, g) \in \RR^m \times \RR^n$ such that
    \begin{equation} \label{eq:warmup:Omega}
        \begin{aligned}
            \forall (i,j) \in \EEE,~
            f_i + g_j &\leq K, \\
            \mu^\top f + \nu^\top g &\geq \theta, \\
             f_1 &= \frac{K}{2}.
        \end{aligned}
    \end{equation}
    Then there exists a normalized minimizer $(f^*, g^*)$ of $\Psi$ that lies in $\Omega$.
    Moreover, $\Omega$ is a bounded convex polytope.
\end{lemma}

\begin{proof}
    Let us show that the constraints
    $\forall (i,j) \in \EEE,~ f_i + g_j \leq K$ are satisfied for any normalized minimizer $(f^*, g^*)$ of $\Psi$.
    Indeed, since $f^* = f[g^*]$, then
    by the same computation as for \autoref{lm:intro:marginals_match},
    \begin{align}
        \forall (i,j) \in \EEE,~~
        e^{[-C_{ij} + f^*_i + g^*_j]/\tau} \mu_i \nu_j
        &\leq \sum_{j'} e^{[-C_{ij'} + f^*_i + g^*_{j'}]/\tau} \mu_i \nu_{j'} = \mu_i \\
        f^*_i + g^*_j 
        &\leq C_{ij} - \tau \log \nu_j
        \leq \max_{(i',j') \in \EEE} C_{i'j'} - \tau \log \nu_{\min}.
    \end{align}
    The symmetric argument using $g^* = g[f^*]$ shows that
    $\forall (i,j) \in \EEE,~ f^*_i + g^*_j \leq \max_\EEE C - \tau \log \mu_{\min}$.

    Now let us show that the constraint $\mu^\top f + \nu^\top g \geq \theta$ is satisfied for any minimizer $(f^*, g^*)$ of $\Psi$.
    Indeed, by definition of $\Psi$,
    \begin{equation}
        \Psi(f^*, g^*) = -\mu^\top f^* - \nu^\top g^*
        \leq \Psi(0, 0) 
        = \tau \log \sum_{ij} e^{-C_{ij}/\tau} \mu_i \nu_j = -\theta.
    \end{equation}

    Finally, note that by \autoref{lm:intro:normalized_minimizers}, if $(f^*, g^*)$ is a normalized minimizer of $\Psi$, then so is $(f^* + b \bmone_m, \allowbreak g^* - b \bmone_n)$ for any $b \in \RR$.
    This implies that there exists a normalized minimizer $(f^*, g^*)$ of $\Psi$ such that $f^*_1 = \frac{K}{2}$, and so, $(f^*, g^*)$ lies in $\Omega$.
    
    The fact that the convex polyhedron $\Omega$ is bounded for any $K, \theta \in \RR$ is proved in \cite[Lemma~5.2]{kalantari2008complexity} independently of any considerations about matrix scaling.
\end{proof}

The following lemma is the adaptation of \cite[proof of Theorem~5.1]{kalantari2008complexity}.
\begin{lemma} \label{lm:warmup:kalantari08_thm51}
    For any $(f, g) \in \Omega$ the set defined in the previous lemma,
    \begin{equation}
        \forall i,~
        \abs{f_i - \frac{K}{2}}
        \leq \frac{K-\theta}{\Delta}
        \qquad \text{and} \qquad
        \forall j,~
        \abs{g_j - \frac{K}{2}}
        \leq \frac{K-\theta}{\Delta}.
    \end{equation}
\end{lemma}

\begin{proof}
    As $\Omega$ is a convex polytope, it suffices to prove the inequalities for all of its vertices.
    Let $(f, g)$ be any vertex of $\Omega$.
    Since $\Omega$ is a polytope in $\RR^{m \times n}$, then at least $m+n$ linearly independent constraints among \eqref{eq:warmup:Omega} must be active at $(f, g)$.
    The equality constraint $f_1 = \frac{K}{2}$ is necessarily active.
    
    \textbf{Case 1:} Suppose that the constraint $\mu^\top f + \nu^\top g \geq \theta$ is not active at $(f, g)$.
    Then, at least $m+n-1$ linearly independent constraints of the form $f_i+g_j \leq K$ for $(i,j) \in \EEE$ must be active at $(f, g)$.
    That is, there exists a set $\EEE' \subset \EEE$ of cardinality $m+n-1$ such that
    $f_i+g_j = K$ for all $(i,j) \in \EEE'$,
    and such that
    $\left\{ (e_i, e_j), (i,j) \in \EEE' \right\}$ is linearly independent where $e_i$ denotes the $i$-th basis vector.
    In particular, the bipartite graph $(\{1 \dots m\} \sqcup \{1 \dots n \}, \EEE')$ cannot contain any cycle, as if it contained a cycle $(i_1, j_1, i_2, j_2, ..., i_\ell, j_\ell, i_1)$ then one would have
    $(e_{i_1}, e_{j_1}) - (e_{i_2}, e_{j_1}) + (e_{i_2}, e_{j_2}) - ... + (e_{i_\ell}, e_{j_\ell}) - (e_{i_1}, e_{j_\ell}) = 0$, contradicting the linear independence.
    Since $(\{1 \dots m\} \sqcup \{1 \dots n \}, \EEE')$ has $m+n-1$ edges and does not contain any cycles, then it is a tree, and so it is connected.

    Now for any $j$ such that $(1,j) \in \EEE'$, i.e., any neighbor $j$ of $1$ in the aforementioned tree, we have
    $g_j = K - f_1 = \frac{K}{2}$.
    For any neighbor $i$ of a neighbor $j$ of $1$, we have $f_i = K - g_j = \frac{K}{2}$.
    Since the tree is connected, recursively applying the argument shows that $\forall i, f_i = \frac{K}{2}$ and $\forall j, g_j = \frac{K}{2}$. 
    In particular, the inequalities claimed in the lemma indeed hold.

    \textbf{Case 2:} Suppose that the constraint $\mu^\top f + \nu^\top g \geq \theta$ is active at $(f, g)$.
    Then, by reasoning as in the previous case, there exists a set $\EEE' \subset \EEE$ of cardinality $m+n-2$ such that $f_i+g_j = K$ for all $(i,j) \in \EEE'$, and such that the bipartite graph $(\{1 \dots m\} \sqcup \{1 \dots n \}, \EEE')$ does not contain any cycles.
    As a consequence, this graph has exactly two connected components, which we denote by $I_1 \sqcup J_1$ and $I_2 \sqcup J_2$. Without loss of generality, suppose $1 \in I_1$.

    By reasoning as in the previous case, we have $\forall i \in I_1, f_i = \frac{K}{2}$ and $\forall j \in J_1, g_j = \frac{K}{2}$. In particular, the inequalities claimed in the lemma hold for those indices.

    For any $(i, j), (i, j') \in \EEE' \cap (I_2 \times J_2)$,
    $f_i = K - g_j = K - g_{j'}$. 
    Since $(I_2 \sqcup J_2, \EEE' \cap (I_2 \times J_2))$ is connected, recursively applying the argument shows that $\forall j \in J_2, g_j = K/2-t$ for some common $t$.
    Moreover $\forall i \in I_2, f_i = K-(K/2-t) = K/2+t$.

    Now recall the active constraint $\mu^\top f + \nu^\top g = \theta$. By the above discussion, it rewrites
    \begin{align}
        \theta = \mu^\top f + \nu^\top g 
        &= \bigg( \sum_{i \in I_1} \mu_i + \sum_{j \in J_1} \nu_j \bigg) \frac{K}{2}
        + \bigg( \sum_{i \in I_2} \mu_i \bigg) \bigg( \frac{K}{2} + t \bigg)
        + \bigg( \sum_{j \in J_2} \nu_j \bigg) \bigg( \frac{K}{2} - t \bigg) \\
        &= K + \bigg( \sum_{i \in I_2} \mu_i - \sum_{j \in J_2} \nu_j \bigg) t.
    \end{align}
    First suppose $K \neq \theta$. Then $\sum_{i \in I_2} \mu_i \neq \sum_{j \in J_2} \nu_j$ and
    \begin{equation}
        \abs{t} = \frac{K - \theta}{\abs{\sum_{i \in I_2} \mu_i - \sum_{j \in J_2} \nu_j}}
        \leq \frac{K-\theta}{\Delta},
    \end{equation}
    since we showed in \autoref{rk:warmup:K_theta} that $K \geq \theta$;
    hence the inequalities claimed in the lemma.
    Now suppose $K = \theta$.
    Then for any $\eps>0$, we may apply the above reasoning to the set $\Omega_\eps$ defined by the same constraints as $\Omega$ but with $\theta$ replaced by $\theta-\eps$.
    This shows that for any 
    $\forall (f, g) \in \Omega_\eps,
    \max_i \abs{f_i-\frac{K}{2}},
    \max_j \abs{g_j-\frac{K}{2}} \leq \frac{\eps}{\Delta}$.
    So by taking a limit $\eps \to 0$, since $\Omega_\eps$ converges to $\Omega$, we can indeed conclude to the inequalities claimed in the lemma.
\end{proof}

We can now conclude the proof of \autoref{prop:warmup:good_minimizer}.

\begin{proof}[Proof of \autoref{prop:warmup:good_minimizer}]
    Let $(f^*, g^*)$ be a normalized minimizer of $\Psi$ that lies in the set $\Omega$ introduced in \autoref{lm:warmup:kalantari08_lm5152}.
    Then by \autoref{lm:warmup:kalantari08_thm51}, 
    $\norm{f^*}_{\var} = \frac12 (\max f^* - \min f^*)
    \leq \frac{K-\theta}{\Delta}$
    and
    $\norm{f^*}_\infty \leq \frac{K}{2} + \frac{K-\theta}{\Delta}$
    and likewise for $g^*$,
    as desired.
\end{proof}

\section{Proof of \autoref{prop:fastasymp:good_approx_minimizer}} \label{sec:apx_delayedpf_GAM}

In this appendix, we present the proof of \autoref{prop:fastasymp:good_approx_minimizer}, restated below.
For ease of presentation, let us first recall the generalized Dulmage-Mendelsohn decomposition from \autoref{prop:fastasymp:DM_decomp} and introduce some notations.

\begin{definition} \label{def:apx_delayedpf_GAM:DM_decomp}
    Suppose that $A$ is asymptotically scalable.
    Let $P$, the partitions $\{ 1 \dots m \} = I_1 \sqcup ... \sqcup I_P$, $\{1 \dots n \} = J_1 \sqcup ... \sqcup J_P$, and the relation ``$\to$'' be defined as in \autoref{prop:fastasymp:DM_decomp}. That is, 
    \begin{itemize}
        \item For all $p \leq P$, 
        $\sum_{i \in I_p} \mu_i = \sum_{j \in J_p} \nu_j$
        and $A^{(p)} = (A_{ij})_{i \in I_p, j \in J_p}$ is $(\restr{\mu}{I_p}, \restr{\nu}{J_p})$-exactly scalable.
        \item The relation ``$\to$'' on $\{1 \dots P\}$ is defined by
        $p \to q \iff \exists i \in I_p, j \in J_q ~\text{s.t.}~ A_{ij}>0$ and $p \neq q$,
        and the directed graph $(\{1 \dots P\}, \{ (p,q);~ p \to q \} )$ is a directed acyclic graph (DAG),
        which is connected provided that $\EEE$ is connected.
    \end{itemize}
    
    We call the subgraphs $(I_p \sqcup J_p, \EEE \cap (I_p \times J_p))$ of $(\{1 \dots m \} \sqcup \{1 \dots n \}, \EEE)$ the \emph{Dulmage-Mendelsohn (DM) components of $(\mu, \nu, \EEE)$}.
    
    We call the graph $(\{1 \dots P\}, \{ (p,q);~ p \to q \} )$ the \emph{DM interaction DAG of $(\mu, \nu, \EEE)$}.
    
    We call the diameter of this DAG, i.e., the maximal length of a path, the \emph{DM diameter of $(\mu, \nu, \EEE)$}.
    %
\end{definition}

The result to be proved in this appendix is the following.
As mentioned in the main text, this is an adapted and refined version of \cite[Lemma~3.3]{allen2017much}.

\begin{proposition*}[\autoref{prop:fastasymp:good_approx_minimizer}, restated]
    Suppose that $A$ is asymptotically scalable.
    Then for any $0 < \eps \leq \tau\, e^{\tau^{-1} (K-\ul\theta) (1+2/\Delta)}$, there exists $\hat g_\eps \in \RR^m$ such that
    \begin{equation}
        \Psi(f[\hat g_\eps], \hat g_\eps) - \inf \Psi \leq \eps
        \qquad \text{and} \qquad 
        \norm{\hat g_\eps}_{\var}
        \leq
        \frac{\tau \ell}{2} \log (\tau/\eps)
        + (K-\ul\theta)
        \left( \frac{\ell}{2} + \frac{1 + \ell}{\Delta} \right)
    \end{equation}
    where $\ell$ is the DM diameter of $(\mu, \nu, \EEE)$ and
    $K, \Delta, \ul\theta$ are the quantities defined in \eqref{eq:warmup:def_K_theta}, \eqref{eq:warmup:def_Delta}, \eqref{eq:warmup:def_ultheta}.
\end{proposition*}

The idea of the proof is to consider good minimizers $(f^{*p}, g^{*p})$ for the exactly-scalable diagonal blocks (\autoref{lm:apx_delayedpf_GAM:diag_blocks}) and to construct $(\hat f, \hat g)$ in the form
$\forall p,
\forall i \in I_p, \hat f_i = f^{*p}_i + t_p$
and
$\forall j \in J_p, \hat g_j = g^{*p}_j - t_p$,
for some offsets $t_p \in \RR$ to be chosen.
This structure allows to easily bound 
$\Psi(\hat f, \hat g) - \inf \Psi$ in terms of the $t_p$
(\autoref{lm:apx_delayedpf_GAM:ub_Psi-infPsi}, \autoref{lm:apx_delayedpf_GAM:tp_tq}),
and it will only remain to choose them appropriately.

\begin{lemma} \label{lm:apx_delayedpf_GAM:diag_blocks}
    For each $p \leq P$, there exists a minimizer $(f^{*p}, g^{*p}) \in \RR^{I_p} \times \RR^{J_p}$ of
    \begin{equation}
        \Psi^p(f^p, g^p)
        = \tau \log \sum_{i \in I_p, j \in J_p} e^{[-C_{ij}+f^p_i+g^p_j]/\tau} \mu_i \nu_j
        - \sum_{i \in I_p} \mu_i f^p_i
        - \sum_{j \in J_p} \nu_j g^p_j
    \end{equation}
    with
    \begin{equation}
        \forall i \in I_p,~~
        \sum_{j \in J_p} e^{[-C_{ij} + f^{*p}_i + g^{*p}_j]/\tau} \mu_i \nu_j = \mu_i \\
        \qquad \text{and} \qquad
        \forall j \in J_p,~~
        \sum_{i \in I_p} e^{[-C_{ij} + f^{*p}_i + g^{*p}_j]/\tau} \mu_i \nu_j = \nu_j
    \end{equation}
    and such that
    \begin{equation}
        \norm{f^{*p}}_{\var}, 
        \norm{g^{*p}}_{\var}
        \leq \frac{K - \ul\theta}{\Delta}
        \qquad\text{and}\qquad
        \norm{f^{*p}}_\infty, 
        \norm{g^{*p}}_\infty
        \leq \frac{K}{2} + \frac{K - \ul\theta}{\Delta}.
    \end{equation}
\end{lemma}

\begin{proof}
    The lemma follows from applying \autoref{prop:warmup:good_minimizer}
    with $\mu, \nu$, and $C$ replaced respectively by 
    $\restr{\mu}{I_p}$,
    $\restr{\nu}{J_p}$, and
    $(C_{ij})_{i \in I_p, j \in J_p}$.
    For the normalization conditions, we use the fact that the minimizer exhibited by \autoref{prop:warmup:good_minimizer} is a normalized minimizer.
    For the estimates on 
    $\norm{f^{*p}}_{\var}, 
    \norm{g^{*p}}_{\var}, 
    \norm{f^{*p}}_\infty, 
    \norm{g^{*p}}_\infty$,
    we use \autoref{rk:warmup:ultheta} to lower-bound the ``$\theta$'' by $\ul\theta$ uniformly in~$p$,
    and we note that the ``$K$'' and the ``$1/\Delta$'' only decrease upon restricting to $I_p, J_p$.
\end{proof}

\begin{lemma} \label{lm:apx_delayedpf_GAM:ub_Psi-infPsi}
    For any $(\hat f, \hat g) \in \RR^m \times \RR^n$, if there exists $L>0$ and $\pi^* \in \Delta_{m \times n}$ such that
    \begin{equation}
        \begin{cases}
            X_\sharp \pi^* = \mu \\
            Y_\sharp \pi^* = \nu
        \end{cases}
        \quad \text{and} \qquad
        \forall i,j,~~
        \pi[\hat f, \hat g]_{ij} \geq \frac{1}{1+L} \pi^*_{ij}
    \end{equation}
    where we recall that
    $\pi[\hat f, \hat g]_{ij} = e^{[-C_{ij}+\hat f_i + \hat g_j]/\tau} \mu_i \nu_j / \hZ$ and
    $\hZ = \sum_{i'j'} e^{[-C_{i'j'}+\hat f_{i'} + \hat g_{j'}]/\tau} \mu_{i'} \nu_{j'}$,
    then
    \begin{equation}
        \Psi(\hat f, \hat g) - \inf \Psi
        \leq \tau \log(1+L) 
        \leq \tau L.
    \end{equation}
\end{lemma}

\begin{proof}
    Denote for concision 
    $\hat\pi = \pi[\hat f, \hat g]$.
    For any $\alpha \in \RR^m, \beta \in \RR^n$,
    by definition of
    $\Psi(f, g) = \tau \log \sum_{ij} e^{[-C_{ij}+\hat f_i + \hat g_j]/\tau} \mu_i \nu_j 
    -\mu^\top f - \nu^\top g$,
    \begin{align}
        \Psi(\hat f+\alpha, \hat g+\beta) - \Psi(\hat f, \hat g)
        &= \tau \log \sum_{ij} \hat\pi_{ij} e^{[\alpha_i + \beta_j]/\tau}
        - \mu^\top \alpha - \nu^\top \beta \\
        &\geq -\tau \log(1+L)
        + \tau \log \sum_{ij} \pi^*_{ij} e^{[\alpha_i + \beta_j]/\tau}
        - \sum_{ij} \pi^*_{ij} (\alpha_i+\beta_j) \\
        &\geq -\tau \log(1+L)
    \end{align}
    by concavity of $\log$ and Jensen's inequality.
    By rearranging the inequality and taking a supremum over $\alpha, \beta$, we obtain 
    $\Psi(\hat f, \hat g) - \inf \Psi
    \leq \tau \log(1+L) \leq \tau L$.
\end{proof}

\begin{lemma} \label{lm:apx_delayedpf_GAM:tp_tq}
    Let $(f^{*p}, g^{*p}) \in \RR^{I_p} \times \RR^{J_p}$ for $p \leq P$ be defined as in \autoref{lm:apx_delayedpf_GAM:diag_blocks}.
    Let any $t_1, ..., t_P \in \RR$ and define $(\hat f, \hat g) \in \RR^m \times \RR^n$ by
    \begin{equation}
        \forall p,~~
        \forall i \in I_p,~ \hat f_i = f^{*p}_i + t_p
        \qquad \text{and} \qquad
        \forall j \in J_p,~ \hat g_j = g^{*p}_j - t_p.
     \end{equation}
     Then
     \begin{equation}
        \Psi(\hat f, \hat g) - \inf \Psi
        \leq \tau\, 
        \exp\left( \tau^{-1} \, \max_{\substack{p, q \\ p \to q}}\, t_p-t_q \right)\,
        e^{\tau^{-1} (K - \ul\theta)(1+2/\Delta)}.
     \end{equation}
\end{lemma}

\begin{proof}
    Define a ``block-diagonal'' reference coupling $\pi^*$ by
    \begin{equation}
        \forall p \leq P,~
        \forall i \in I_p, j \in J_p,~~
        \pi^*_{ij} = e^{[-C_{ij} + f^{*p}_i + g^{*p}_j]/\tau} \mu_i \nu_j
    \end{equation}
    and $\pi^*_{ij} = 0$ for all $i \in I_p, j \in J_q$ with $p \neq q$.
    Then for all $p \leq P$, by definition of the $f^{*p}, g^{*p}$,
    \begin{equation}
        \forall i \in I_p,~~
        \sum_j \pi^*_{ij} = \sum_{j \in J_p} \pi^*_{ij} = \mu_i
        \qquad\text{and}\qquad
        \forall j \in J_p,~~
        \sum_i \pi^*_{ij} = \sum_{i \in I_p} \pi^*_{ij} = \nu_j.
    \end{equation}
    Hence $X_\sharp \pi^* = \mu$ and $Y_\sharp \pi^* = \nu$. 
    
    Moreover, denote for concision
    $\hZ = \sum_{ij} e^{[-C_{ij}+\hat f_{i} + \hat g_{j}]/\tau} \mu_{i} \nu_{j}$.
    Then for any $i \leq m, j \leq n$, say $i \in I_p$ and $j \in J_q$,
    \begin{itemize}
        \item if $p=q$ then 
        $\pi[\hat f, \hat g]_{ij} = \hZ^{-1} e^{[-C_{ij}+\hat f_i + \hat g_j]/\tau} \mu_i \nu_j
        = \hZ^{-1} e^{[-C_{ij}+f^{*p}_i + g^{*p}_j]/\tau} \mu_i \nu_j
        = \hZ^{-1} \pi^*_{ij}$,
        and
        \item if $p \neq q$ then 
        $\pi[\hat f, \hat g]_{ij} \geq 0 = \pi^*_{ij}$.
    \end{itemize}
    Hence $\forall i, j,~ \hat\pi_{ij} \geq \hZ^{-1} \pi^*_{ij}$.
    Explicitly, by distinguishing edges within and across DM components,
    \begin{align}
        \hZ
        &= \sum_{ij} e^{[-C_{ij}+\hat f_{i} + \hat g_{j}]/\tau} \mu_{i} \nu_{j} \\
        &= \sum_{p \leq P}
        ~\sum_{i \in I_p, j \in J_p}~
        e^{[-C_{ij} + f^{*p}_i + g^{*p}_j]/\tau} \mu_i \nu_j
        + \sum_{\substack{p, q \leq P \\ p \to q}}
        ~\sum_{i \in I_p, j \in J_q}~
        e^{[-C_{ij} + f^{*p}_i + t_p + g^{*q}_j - t_q]/\tau} \mu_i \nu_j \\
        &= ~~\underbrace{
            \sum_p \sum_{i \in I_p} \mu_i
        }_{1}~~
        + \sum_{\substack{p, q \\ p \to q}}
        e^{(t_p - t_q)/\tau}
        \sum_{i \in I_p, j \in J_q}
        e^{[-C_{ij} + f^{*p}_i + g^{*q}_j]/\tau} \mu_i \nu_j.
    \end{align}
    
    So we can apply \autoref{lm:apx_delayedpf_GAM:ub_Psi-infPsi} with $L = \hZ - 1$,
    yielding
    \begin{equation}
        \Psi(\hat f, \hat g) - \inf \Psi
        \leq \tau\, 
        \sum_{\substack{p, q \\ p \to q}}
        e^{(t_p - t_q)/\tau}
        \sum_{i \in I_p, j \in J_q}
        e^{[-C_{ij} + f^{*p}_i + g^{*q}_j]/\tau} \mu_i \nu_j.
    \end{equation}
    Now by definition of the $f^{*p}, g^{*p}$ and the estimate from \autoref{lm:apx_delayedpf_GAM:diag_blocks},
    \begin{equation}
        \forall i \in I_p, j \in J_q,~~
        -C_{ij} + f^{*p}_i + g^{*q}_j
        \leq -\min_{i'j'} C_{i'j'} + \norm{f^{*p}}_\infty + \norm{g^{*p}}_\infty
        \leq -\ul\theta + K + 2 \frac{K-\ul\theta}{\Delta}.
    \end{equation}
    Substituting this into the previous inequality, we obtain
    \begin{align}
        \Psi(\hat f, \hat g) - \inf \Psi
        &\leq \tau\, 
        \sum_{\substack{p, q \\ p \to q}}
        e^{(t_p - t_q)/\tau} \,
        e^{\tau^{-1} (K - \ul\theta)(1+2/\Delta)}
        \sum_{\substack{i \in I_p, j \in J_q \\ (i,j) \in \EEE}}
        \mu_i \nu_j \\
        &\leq \tau\, 
        \exp\left( \tau^{-1}\, \max_{\substack{p, q \\ p \to q}}\, t_p-t_q \right)\,
        e^{\tau^{-1} (K - \ul\theta)(1+2/\Delta)}
        ~
        \underbrace{
            \sum_{\substack{p, q \\ p \to q}}
            \sum_{\substack{i \in I_p, j \in J_q \\ (i,j) \in \EEE}}
            \mu_i \nu_j
        }_{\leq 1}
    \end{align}
    as announced.
\end{proof}

We can now conclude the proof of \autoref{prop:fastasymp:good_approx_minimizer}.

\begin{proof}[Proof of \autoref{prop:fastasymp:good_approx_minimizer}]
    For each $p \leq P$, let $\ell_p$ denote the maximal length of a path ending at $p$ in the DM interaction DAG. Note that for any edge $p \to q$ in the DAG, it holds 
    $\ell_q \geq \ell_p + 1$.
    
    Let $(f^{*p}, g^{*p}) \in \RR^{I_p} \times \RR^{J_p}$ for $p \leq P$ be defined as in \autoref{lm:apx_delayedpf_GAM:diag_blocks},
    and let
    \begin{equation}
        t_p = \rho \, \ell_p
        \qquad\text{where}\qquad
        \rho = \tau \log(\tau/\eps) + (K-\ul\theta) \left( 1 + \frac{2}{\Delta} \right).
    \end{equation}
    Note that $\rho \geq 0$ provided that
    $\eps \leq \tau\, e^{\tau^{-1} (K-\ul\theta) (1+2/\Delta)}$, which we assumed in the proposition statement.
    In particular,
    \begin{equation}
        \forall p, q \leq P ~\text{s.t.}~ p \to q,~~
        t_p - t_q
        = \rho (\ell_p - \ell_q)
        \leq -\rho
    \end{equation}
    and $\min_p t_p = 0$ and $\max_p t_p = \rho \, \ell$ where we recall that $\ell$ is the DM diameter.
    
    Let $(\hat f_\eps, \hat g_\eps) \in \RR^m \times \RR^n$ be defined from the $f^{*p}, g^{*p}$, and $t_p$ as in \autoref{lm:apx_delayedpf_GAM:tp_tq}. 
    Then
    \begin{align}
        \norm{\hat f_\eps}_{\var}
        &\leq \max_p \norm{f^{*p}}_{\var} + \frac12 ((\max\nolimits_p t_p) - (\min\nolimits_p t_p)) \\
        &\leq \frac{K-\ul\theta}{\Delta} + \frac{\rho\, \ell}{2} 
        = \frac{K-\ul\theta}{\Delta} 
        + \frac{\ell}{2} \left( 
            \tau \log(\tau/\eps) + (K-\ul\theta) \left( 1 + \frac{2}{\Delta} \right) 
        \right) \\
        &= \frac{\tau \ell}{2} \log (\tau/\eps)
        + (K-\ul\theta)
        \left( \frac{\ell}{2} + \frac{1 + \ell}{\Delta} \right)
    \end{align}
    and likewise for $\norm{\hat g_\eps}_{\var}$.
    Moreover,
    \begin{align}
        \Psi(f[\hat g_\eps], \hat g_\eps) - \inf \Psi
        \leq \Psi(\hat f_\eps, \hat g_\eps) - \inf \Psi
        &\leq \tau\, 
        \exp\left( \tau^{-1} \, \max_{\substack{p, q \\ p \to q}}\, t_p-t_q \right)\,
        e^{\tau^{-1} (K - \ul\theta)(1+2/\Delta)} \\
        &\leq \tau ~ e^{-\tau^{-1} \rho} ~ e^{\tau^{-1} (K - \ul\theta)(1+2/\Delta)}
        = \eps,
    \end{align}
    as desired.
\end{proof}

\section{Proof of \autoref{lm:fastasymp:technical_ub_inf}} \label{sec:apx_delayedpf_technical}

\begin{lemma*}[\autoref{lm:fastasymp:technical_ub_inf}, restated]
    Let $b, X>0$ and $M \in \RR$ such that $X < e^{M/b}$. Then for any
    \begin{equation}
        0 < \alpha \leq
        \frac{X}{2 e \, b^2 \, \max\left\{1, \log(e^{M/b}/X) \right\}},
    \end{equation}
    we have
    \begin{equation}
        \inf_{0<x\le X}
        x + \alpha \Big( b \log(1/x) + M \Big)^2
        \leq 4 \alpha b^2
        \left[
            \log\left( \frac{e^{M/b}}{2\alpha b^2} \right)
        \right]^2.
    \end{equation}
\end{lemma*}

\begin{proof}
    Let
    \begin{equation}
        x(\alpha) = 2\alpha b^2 \log\left(\frac{e^{M/b}}{2\alpha b^2}\right).
    \end{equation}
    Since $\alpha \leq X/(2eb^2)$ and $X<e^{M/b}$, we have 
    $\log \frac{e^{M/b}}{2\alpha b^2} \geq 1$ and $x(\alpha)>0$.

    Let us show that $x(\alpha) \leq X$.
    One can check that $\alpha \mapsto x(\alpha)$
    is increasing on $(0, \frac{e^{M/b}}{2eb^2}]$.
    So if $\log(e^{M/b}/X) \leq 1$, then
    since $\frac{X}{2eb^2} \leq \frac{e^{M/b}}{2eb^2}$,
    \begin{equation}
        x(\alpha)
        \leq x\left( \frac{X}{2eb^2} \right)
        = \frac{X}{e} \log\left( \frac{e^{M/b}}{X/e} \right)
        = \frac{X}{e} \left( 1+\log(e^{M/b}/X) \right)
        \leq \frac{2X}{e}
        \leq X.
    \end{equation}
    Likewise, if $\log(e^{M/b}/X) > 1$, then
    $\frac{X}{2eb^2 \log(e^{M/b} / X)} \leq \frac{e^{M/b}}{2eb^2}$
    since $\forall u > 1, \frac1u \leq e^u$,
    and so
    \begin{align}
        x(\alpha)
        &\leq
        x\left( \frac{X}{2e b^2 \log(e^{M/b}/X)} \right) \\
        &= \frac{X}{e \log(e^{M/b}/X)}
        \log\left(
            \frac{e\, e^{M/b}\log(e^{M/b}/X)}{X}
        \right) \\
        &= \frac{X}{e \log(e^{M/b}/X)}
        \left[
            1 + \log(e^{M/b}/X) + \log\log(e^{M/b}/X)
        \right]
        \leq X
    \end{align}
    since 
    $\forall u > 1, \frac{1+u+\log u}{e u} \leq 1$.
    
    Therefore $0 < x(\alpha) \leq X$, so the infimum in the lemma statement can be upper-bounded by evaluating it at $x(\alpha)$.
    Now
    \begin{equation}
        M + b\log(1/x(\alpha))
        = M - b \log\left(
            2\alpha b^2
            \log\left(\frac{e^{M/b}}{2\alpha b^2}\right)
        \right)
        = b\left[
            \log \left( \frac{e^{M/b}}{2\alpha b^2} \right)
            - \log\log \left( \frac{e^{M/b}}{2\alpha b^2} \right)
        \right].
    \end{equation}
    Hence
    \begin{align}
        x(\alpha) + \alpha \Big( M + b \log(1/x(\alpha)) \Big)^2
        &= \alpha b^2
        \left\{
            2\log\left(\frac{e^{M/b}}{2\alpha b^2}\right)
            + \left[
                \log\left(\frac{e^{M/b}}{2\alpha b^2}\right)
                -
                \log\log\left(\frac{e^{M/b}}{2\alpha b^2}\right)
            \right]^2
        \right\} \\
        &\leq
        4\alpha b^2
        \left[
            \log\left(\frac{e^{M/b}}{2\alpha b^2}\right)
        \right]^2
    \end{align}
    where for the last inequality we used the fact that
    $\forall u \geq 1,~
    2 u + (u - \log u)^2 \leq 4 u^2$
    applied to
    ${u = \log \frac{e^{M/b}}{2\alpha b^2} \geq 1}$.
\end{proof}

\begin{remark}
    The bound in \autoref{lm:fastasymp:technical_ub_inf} is asymptotically tight up to a constant factor for small $\alpha$.
\end{remark}

\fi

\end{document}